\documentclass[11pt]{amsart}
\usepackage{latexsym}
\usepackage{amsfonts}
\usepackage{amssymb}
\usepackage{slashed}
\usepackage{amsmath}
\usepackage{mathrsfs}
\usepackage{hyperref}
\usepackage{tikz}

\usetikzlibrary{matrix,arrows}

\newcommand{\AAZ}{{\mathcal A}_{\hbar{}}^{\infty}}

\title[Subadditivity and additivity of Yang-Mills]{Subadditivity and additivity of the Yang-Mills action
functional in noncommutative geometry}

\author{Satyajit Guin}
\address{Indian Institute of Science Education and Research, Mohali, Punjab 140306}
\email{satyamath@gmail.com\,, satyajit@iisermohali.ac.in}

\keywords{Yang-Mills functional, subadditivity of Yang-Mills, additivity of Yang-Mills, critical points of Yang-Mills, connection, curvature}

\date{\today}
\subjclass[2010]{58B34, 81T75}

\newtheorem{definition}{Definition}[section]
\newtheorem{theorem}[definition]{Theorem}
\newtheorem{lemma}[definition]{Lemma}
\newtheorem{proposition}[definition]{Proposition}
\newtheorem{corollary}[definition]{Corollary}

\newtheorem{remark}[definition]{Remark}

\begin{document}

\begin{abstract}
We formulate notions of subadditivity and additivity of the Yang-Mills action functional in noncommutative geometry. We identify a suitable
hypothesis on spectral triples which proves that the Yang-Mills functional is always subadditive, as per expectation. The additivity property is
much stronger in the sense that it implies the subadditivity property. Under this hypothesis we obtain a necessary and sufficient condition for
the additivity of the Yang-Mills functional. An instance of additivity is shown for the case of noncommutative $n$-tori. We also investigate the
behaviour of critical points of the Yang-Mills functional under additivity. At the end we discuss few examples involving compact spin manifolds,
matrix algebras, noncommutative $n$-torus and the quantum Heisenberg manifolds which validate our hypothesis.
\end{abstract}

\maketitle
\medskip

\section{Introduction}

Given a vector bundle $\mathbb{E}$ on a Riemannian manifold $(M,g)$ and a connection $\nabla$ on $\mathbb{E}$, the Yang-Mills functional
is given by
\begin{eqnarray}\label{classical Y-M}
\mathcal{YM}(\nabla) & = & \int_\mathbb{M}||\varTheta_\nabla||^2dV_g\,\,,
\end{eqnarray}
where $\varTheta_\nabla$ denotes the curvature of $\nabla$. Atiyah-Bott (\cite{AB}) initiated the study of corresponding gradient flow to the
Yang-Mills energy on a closed Riemann surface and proposed studying it as a means of understanding the topology of the space of connections
using infinite dimensional Morse theory. Immediately, one remarkable application of the flow appeared in Donaldson's characterization of the
correspondence between the algebraic and differential geometry on K\"ahler manifolds (\cite{Ds}). He demonstrated that the stability of a bundle
is equivalent to it admitting irreducible Hermitian-Einstein connection with respect to the K\"ahler metric. Around same time noncommutative
differential geometry was invented by A. Connes (\cite{Con0},\cite{Con1}) for the purpose of extending differential geometry and topology beyond
their classical framework in order to deal with `spaces', such as leaf spaces of foliations and orbit spaces of discrete or Lie group actions on
manifolds, which elude analysis by classical methods. The generalization of the Yang-Mills functional to the noncommutative context first appeared
in (\cite{CoR}) by the work of Connes-Rieffel for the case of $C^*$-dynamical systems. Latter Connes formulated this notion more formally in the
language of $\mathcal{K}$-cycles or spectral triples in (\cite{Con2}), and investigated the case of noncommutative two-torus in great detail which
suggests extensions of Yang-Mills theoretic techniques in the study of noncommutative differential (and possibly holomorphic) geometry of `vector
bundles' on $C^*$-algebras. It turns out that these two notions of Yang-Mills in noncommutative geometry, the older one for the $C^*$-dynamical
systems (due to Connes-Rieffel in \cite{CoR}) and the more formal one in the context of spectral triples (due to Connes in \cite{Con2}), are
equivalent for the case of noncommutative $n$-tori (\cite{CG1}) and the quantum Heisenberg manifolds (\cite{CG2}). However, the general case
remains unanswered. But certainly, the formulation of Yang-Mills in the spectral triple setting is the adequate generalization of the classical
Yang-Mills to the noncommutative framework.

In the Noncommutative Geometry programme of Connes, by a noncommutative topological space we mean an involutive subalgebra of a (unital)
$C^*$-algebra. It is now widely accepted that geometry over a noncommutative space $\mathcal{A}$ is governed by a triple $(\mathcal{A},
\mathcal{H},D)$, called spectral triple. Here, $\mathcal{A}$ is a unital associative $\star$-subalgebra of a $C^*$-algebra $\mathscr{A}$
faithfully represented on the separable Hilbert space $\mathcal{H}$, and $D$ is an unbounded self-adjoint operator with compact resolvent
acting on $\mathcal{H}$ such that $[D,a]$ extends to a bounded operator on $\mathcal{H}$ for all $a\in\mathcal{A}$. If there exists a
$\mathbb{Z}_2$-grading operator $\gamma\in\mathcal{B}(\mathcal{H})$ which commutes with $\mathcal{A}$ and anticommutes with $D$, then the
quadruple $(\mathcal{A},\mathcal{H},D,\gamma)$ is called an {\it even spectral triple}. Spectral triple generalizes classical spin manifolds
to the noncommutative framework. Here, finitely generated projective modules equipped with Hermitian structure serve the role of complex
vector bundles and the $L^2$-norm is specified by the Dixmier trace on spectral triples. The Yang-Mills action functional (\cite{Con3}) on
a finitely generated projective (f.g.p) module $\mathcal{E}$ over $\mathcal{A}$, equipped with a Hermitian structure, is a certain map
$\mathcal{YM}:C(\mathcal{E})\longrightarrow\mathbb{R}_{\geq 0}$ generalizing (\ref{classical Y-M}), where $C(\mathcal{E})$ denotes the affine
space of compatible connections on $\mathcal{E}$. Here, compatibility is described with respect to the Hermitian structure on f.g.p module
$\mathcal{E}$. A crucial application in physics is observed in (\cite{CDS}). Note that Yang-Mills represents energy functional and hence,
critical points of it is of particular interest in mathematics as well as in physics literature. These have been investigated by Rieffel
(\cite{Rfl2}) on the noncommutative two-torus and Kang (\cite{Kan}) on the quantum Heisenberg manifolds. It would not be an exaggeration
to say that Yang-Mills is an important and active area of research in noncommutative geometry, and over the years it has been studied by
various authors (e.g. \cite{CoR},\cite{Rfl2},\cite{Sp},\cite{Kan},\cite{Lee},\cite{CG1},\cite{CG2}).

Since the domain of the Yang-Mills functional (henceforth briefly abbreviated as Y-M) is an affine space, the usual notions of subadditivity
and additivity of a function do not make sense. We systematically formulate these notions and prove that under a suitable hypothesis on spectral
triples Y-M is always subadditive. This is expected since Y-M represents energy functional. The notion of additivity turns out to be stronger
than subadditivity in the sense that additivity implies subadditivity. Let us briefly describe our setting. Like in the classical case where
forming the product between two geometric spaces is a basic operation in geometry, considering tensor product of noncommutative spaces is also
of much relevant importance not only for construction of a would-be tensor category, but also bears interest for some applications in theoretical
physics (\cite{CM}). For example, the almost commutative spectral triple corresponding to the standard model of particle physics (\cite{CC})
is a tensor product of a canonical commutative spectral triple with a finite-dimensional noncommutative one. Given two even spectral triples
$(\mathcal{A}_j,\mathcal{H}_j,D_j,\gamma_j),\,j=1,2$, their product is defined by the following rule, due to Connes (\cite{Con2}),
\begin{center}
$(\mathcal{A}_1,\mathcal{H}_1,D_1,\gamma_1)\otimes(\mathcal{A}_2,\mathcal{H}_2,D_2,\gamma_2):=(\mathcal{A}_1\otimes \mathcal{A}_2\,,\,\mathcal{H}_1
\otimes\mathcal{H}_2\,,\,D=D_1\otimes 1+\gamma_1\otimes D_2\,,\,\gamma_1\otimes\gamma_2)\,.$
\end{center}
It is enough if one of the spectral triples, instead of both, is even. However, in our context, w.l.o.g. we always consider even spectral
triples. This is explained at the beginning of Section $[3]$. Now, if $\mathcal{E}_1$ and $\mathcal{E}_2$ are two Hermitian f.g.p modules over
$\mathcal{A}_1$ and $\mathcal{A}_2$ respectively, and $\nabla_j\in C(\mathcal{E}_j)$ for $j=1,2$, then $\nabla:=\nabla_1\otimes 1+1\otimes
\nabla_2$ is a connection on $\,\mathcal{E}=\mathcal{E}_1\otimes\mathcal{E}_2$. Important observation is that it is a compatible connection,
i,e. $\nabla\in C(\mathcal{E})$, with respect to a natural Hermitian structure on $\mathcal{E}$. We use the structure theorem of Hermitian
f.g.p modules obtained in (\cite{CG1}) to prove this. A natural question is whether there is any relation between $\mathcal{YM}(\nabla)$ and
$\mathcal{YM}(\nabla_j)$ for $j=1,2$. We define the notions of subadditivity and additivity of Y-M in this context. Under the following
hypothesis on spectral triples $(\mathcal{A}_j,\mathcal{H}_j,D_j,\gamma_j),\,j=1,2$,
\medskip

\textbf{Hypothesis~:} $\frac{\pi_1(\Omega^2(\mathcal{A}_1))\otimes\mathcal{A}_2+\mathcal{A}_1\otimes\pi_2(\Omega^2(\mathcal{A}_2))}{\pi_1(d_1J_0^1
(\mathcal{A}_1))\otimes\mathcal{A}_2+\mathcal{A}_1\otimes\pi_2(d_2J_0^1(\mathcal{A}_2))}\cong\Omega_{D_1}^2(\mathcal{A}_1)\otimes\mathcal{A}_2
\bigoplus\mathcal{A}_1\otimes\Omega_{D_2}^2(\mathcal{A}_2)$ as $\mathcal{A}_1\otimes\mathcal{A}_2$-bimodules.
\medskip

we prove that Y-M is always subadditive, as per expectation. To validate this hypothesis we discuss few examples involving compact spin manifolds,
matrix algebras, noncommutative $n$-tori and the quantum Heisenberg manifolds. These are the cases for which the respective Dirac dga is
known in the literature. Under the above hypothesis we also obtain a necessary and sufficient condition for additivity of Y-M. An instance of
additivity of Y-M is shown for the case of noncommutative $n$-tori. It is also natural to ask if Y-M becomes additive then how its critical
points behave. For this we obtain a useful necessary and sufficient condition which determines when the critical points of Y-M for two spectral
triples give rise to a critical point of Y-M for the product spectral triple.

Organization of the paper is as follows. Section $[2]$ is mainly preliminaries. Sections $[3]$ is the main content where we define the notions
of subadditivity and additivity of Y-M and prove the above discussed results. Section $[4]$ contains an instance of additivity of Y-M. Section
$[5],\,[6],\,[7]$ discuss examples  where our hypothesis is validated. These include the case of compact spin manifolds, matrix algebras,
noncommutative $n$-torus and the quantum Heisenberg manifolds.
\bigskip


\section{Spectral triples and the Yang-Mills functional}

All algebras considered in this article will be assumed unital.

\begin{definition}
A spectral triple $(\mathcal{A},\mathcal{H},D)$ over a unital, associative $\star$-algebra $\mathcal{A}$ consists of the following$\,:$
\begin{enumerate}
\item a $\star$-representation $\pi$ of $\mathcal{A}$ on the separable Hilbert space $\mathcal{H}$,
\item an unbounded self-adjoint operator $D$ acting on $\mathcal{H}$ such that
\begin{enumerate}
\item[(i)] $[D,\pi(a)]$, initially defined on $Dom(D)$, extends to a bounded operator on $\mathcal{H}$,
\item[(ii)] $D$ has compact resolvent.
\end{enumerate}
\end{enumerate}
\end{definition}

It is said to be an {\it even spectral triple} if there exists a $\mathbb{Z}_2$-grading $\gamma\in\mathcal{B}(\mathcal{H})$ (i,e. $\gamma=\gamma^*$
and $\gamma^2=id$) such that $\gamma$ commutes with each element of $\mathcal{A}$ and anticommutes with $D$. If no such $\gamma$ is present then
the spectral triple $(\mathcal{A},\mathcal{H},D)$ is called {\it odd}. We will always assume that $\pi$ is a unital faithful representation.

\begin{definition}
If $\,|D|^{-d}$, for positive number $d$, lies in the Dixmier ideal $\,\mathcal{L}^{(1,\infty)}\subseteq\mathcal{B}(\mathcal{H})$, then the
spectral triple $(\mathcal{A},\mathcal{H},D)$ is called $d$-summable spectral triple $($this is sometimes referred as  $(d,\infty)$ or
$d^+$-summable in the literature$)$.
\end{definition}

Associated to every spectral triple $(\mathcal{A},\mathcal{H},D)$ there is a differential graded algebra (dga) $(\Omega^\bullet_D(\mathcal{A}),d\,)$
defined by Connes, which we will call the Connes' calculus or the Dirac DGA. Recall its definition from (Ch. $[6]$ in \cite{Con2}). However,
for our purpose, the following is enough
\begin{align*}
\Omega_D^1(\mathcal{A}) &= \{\,\sum a_j[D,b_j]\,:\,a_j,b_j\in\mathcal{A}\,\}\,\subseteq\,\mathcal{B}(\mathcal{H})\\
\pi(\Omega^2(\mathcal{A})) &= \{\,\sum a_j[D,b_j][D,c_j]\,:\,a_j,b_j,c_j\in\mathcal{A}\,\}\,\subseteq\,\mathcal{B}(\mathcal{H})\\
\pi(dJ_0^1(\mathcal{A})) &= \{\,\sum [D,b_j][D,c_j]\,:\,\sum b_j[D,c_j]=0\,,b_j,c_j\in\mathcal{A}\}\,\subseteq\,\mathcal{B}(\mathcal{H})\\
\Omega_D^2(\mathcal{A}) &= \pi(\Omega^2(\mathcal{A}))/\pi(dJ_0^1(\mathcal{A}))\,.
\end{align*}
All of these are bimodules over $\mathcal{A}$. We have the Dirac dga differentials $d:\mathcal{A}\longrightarrow\Omega_D^1(\mathcal{A})$ given
by $a\longmapsto[D,a]$, and $d:\Omega_D^1(\mathcal{A})\longrightarrow\Omega_D^2(\mathcal{A})$ given by $a[D,b]\longmapsto[\,[D,a][D,b]\,]$. Note
that $(da)^*=-d(a^*)$ by convention. For the classical case of compact spin manifolds, where $D$ is the classical Dirac operator, $\Omega^\bullet_D$
gives back the space of de-Rham forms (Page $551$ in \cite{Con2}). So, Dirac dga can be thought of as noncommutative space of forms. However, this
dga is very hard to compute and not much of computation is known in the literature except (\cite{CS},\cite{CG1},\cite{CG2},\cite{CG3}). Using
this Dirac dga, Connes extended the classical notion of Yang-Mills action functional to the noncommutative geometry framework in (\cite{Con2}).
Let us recall it now.
\medskip

Let $\mathcal{E}$ be a finitely generated projective right module over $\mathcal{A}$. We will write f.g.p to mean finitely generated projective
throughout the article. Unless explicitly mentioned, we will only consider right modules in this article. Let $\mathcal{E}^*:=\mathcal{H}om_\mathcal{A}
(\mathcal{E},\mathcal{A})$. Clearly, $\mathcal{E}^*$ is also a right $\mathcal{A}$-module by the rule $\,(\phi\,.\,a)\,(\eta):=a^*\phi(\eta),
\,\forall\,\eta\in\mathcal{E},a\in\mathcal{A}$.

\begin{definition}\label{def}$($\cite{Con2}$)$
A {\it Hermitian} structure on $\mathcal{E}$ is an $\mathcal{A}$-valued positive-definite sesquilinear map $\langle\,.\,,.\,\rangle_{\mathcal{A}}$
satisfying the following $:$
\begin{enumerate}
\item [(a)] $\langle\xi,\xi'\rangle_{\mathcal{A}}^*=\langle\xi',\xi\rangle_{\mathcal{A}}\,,\quad\forall\,\xi,\xi'\in\mathcal{E}$.
\item [(b)] $\langle\xi,\xi'.\,a\rangle_{\mathcal{A}}=(\langle\xi,\xi'\rangle_{\mathcal{A}})a\,,\quad\forall\,\xi,\xi'\in\mathcal{E},\,\,\forall\,a
\in\mathcal{A}$.
\item [(c)] The map $\,\xi\longmapsto\Phi_\xi$ from $\mathcal{E}$ to $\mathcal{E}^*$, given by $\Phi_\xi(\eta)=\langle\xi,\eta\rangle_\mathcal{A}$
for all $\eta\in\mathcal{E}$, gives conjugate linear $\mathcal{A}$-module isomorphism between $\mathcal{E}$ and $\mathcal{E}^*$. This property
will be referred as the self-duality of $\mathcal{E}$.
\end{enumerate}
\end{definition}
Any free $\mathcal{A}$-module $\mathcal{E}_0=\mathcal{A}^n$ has a Hermitian structure on it, given by $\langle\,\xi,\eta\,\rangle_\mathcal{A}=
\sum_{j=1}^n\xi_j^*\eta_j$ for all $\,\xi=(\xi_1,\ldots,\xi_n)\,,\,\eta=(\eta_1,\ldots,\eta_n)\in\mathcal{E}_0$. We refer this as the {\it canonical
Hermitian structure} on $\mathcal{E}_0$. By definition, any f.g.p module $\mathcal{E}$ can be written as $\mathcal{E}=p\mathcal{A}^n$ for some
idempotent $p\in M_n(\mathcal{A})$. If this idempotent $p$ is a projection, i,e. $p=p^2=p^*$, one can restrict the canonical Hermitian structure on
$\mathcal{A}^n$ to $\mathcal{E}$ and then $\mathcal{E}$ becomes a Hermitian f.g.p module. Moreover, it is proved in Lemma $2.2(b)$ in (\cite{CG1})
that if $\mathcal{A}$ is stable under the holomorphic functional calculus in a unital $C^*$-algebra $\mathscr{A}$ (in which case the unit will
belong to $\mathcal{A}$), then we have the following existence lemma of Hermitian structure.

\begin{lemma}[\cite{CG1}]
Every f.g.p module $\mathcal{E}$ over $\mathcal{A}$ is isomorphic as a f.g.p module with $p\mathcal{A}^n$ where $p\in M_n(\mathcal{A})$ is a
self-adjoint idempotent, that is a projection. Hence, $\mathcal{E}$ has a Hermitian structure on it.
\end{lemma}

\begin{remark}\rm
The above lemma proves the existence of Hermitian structure with the assumption of closure under holomorphic functional calculus. Without the
assumption, the existence of Hermitian structure on arbitrary f.g.p module over $\mathcal{A}$ is not known.
\end{remark}

With the assumption of closure under the holomorphic functional calculus, one has the following structure theorem of Hermitian f.g.p module.
(Th. $3.3$ in \cite{CG1}).

\begin{theorem}[\cite{CG1}]\label{structure thm of herm. mod.}
Let $\mathcal{E}$ be a f.g.p $\mathcal{A}$-module with a Hermitian structure on it. Suppose $\mathcal{A}$ is stable under the holomorphic
functional calculus in a $C^*$-algebra $\mathscr{A}$. Then we have a self-adjoint idempotent $\,p\in M_n(\mathcal{A})$ such that $\mathcal{E}
\cong p\mathcal{A}^n$ as f.g.p module, and $\mathcal{E}$ has the induced canonical Hermitian structure.
\end{theorem}

In his book (\cite{Con2}), Connes has suggested that in the context of Hermitian structure and the Yang-Mills functional one should always work
with spectrally invariant algebras, that is subalgebras of $C^*$-algebras stable under the holomorphic functional calculus. The reason is that
all possible notions of positivity will coincide in that case. Moreover, we will also have Th. (\ref{structure thm of herm. mod.}) which makes
computation with the Hermitian structure much easier. Hence, incorporating Connes' suggestion, throughout the article we will always work with
spectrally invariant algebras. Note that in the classical situation of manifolds, $C^\infty(\mathbb{M})$ is indeed spectrally invariant subalgebra
of $C(\mathbb{M})$ (\cite{GVF}).

\begin{definition}
Let $\mathcal{E}$ be a f.g.p module over $\mathcal{A}$ equipped with a Hermitian structure $\langle\,.,.\,\rangle_\mathcal{A}$. A compatible
connection on $\mathcal{E}$ is a $\mathbb{C}$-linear map $\nabla:\mathcal{E}\longrightarrow\mathcal{E}\,\otimes_\mathcal{A}\Omega _{D}^1
(\mathcal{A})$ satisfying
\begin{enumerate}
\item[(a)] $\nabla(\xi a)=(\nabla\xi)a+\xi\otimes da,\quad\forall\,\xi\in\mathcal{E},\,a\in\mathcal{A};$
\item[(b)] $\langle\,\xi,\nabla\eta\,\rangle-\langle\,\nabla\xi,\eta\,\rangle=d\langle\,\xi,\eta\,\rangle_\mathcal{A}\quad\forall\,\xi,\eta\in
\mathcal{E}\,\,\,\,($Compatibility$)$.
\end{enumerate}
\end{definition}

The meaning of the last equality in $\Omega_D^1(\mathcal{A})$ is, if $\nabla(\xi)=\sum\xi_j\otimes\omega_j\in\mathcal{E}\otimes\Omega_D^1
(\mathcal{A})$, then $\langle\nabla\xi,\eta\rangle=\sum\omega_j^*\langle\xi_j,\eta\rangle_\mathcal{A}$. Any f.g.p right module has a connection.
An example of a compatible connection is the {\it Grassmannian connection} $\nabla_0$ on $\mathcal{E}=p\mathcal{A}^n$, given by $\nabla_0(\xi)=
pd\xi$, where $d\xi=(d\xi_1,\ldots,d\xi_n)$ and $p\in M_n(\mathcal{A})$ is a projection. This connection is compatible with the induced canonical
Hermitian structure on $\mathcal{E}$. The set of all compatible connections on $\,\mathcal{E}$, which we denote by $C(\mathcal{E})$, is an affine
space with associated vector space $\mathcal{H}om_\mathcal{A}(\mathcal{E},\mathcal{E}\otimes_\mathcal{A}\Omega_D^1(\mathcal{A}))$ (\cite{Con2}).
Any connection $\nabla$ uniquely extends to a $\mathbb{C}$-linear map
\begin{center}
$\widetilde{\nabla}:\mathcal{E}\otimes_\mathcal{A}\Omega_D^1(\mathcal{A})\longrightarrow\mathcal{E}\otimes_\mathcal{A}\Omega_D^2(\mathcal{A})$
\end{center}
satisfying
\begin{eqnarray*}
\widetilde{\nabla}(\xi\otimes\omega)=(\nabla\xi)\omega+\xi\otimes d\omega,\quad\forall\,\xi\in\mathcal{E},\,\,\omega\in\Omega_D^1(\mathcal{A}).
\end{eqnarray*}
It can be easily checked that $\widetilde\nabla$, defined above, satisfies the Leibniz rule, i,e.
\begin{eqnarray*}
\widetilde{\nabla}(\eta a)=\widetilde{\nabla}(\eta)a-\eta\tilde{d}a\,,\,\,\forall\,a\in\mathcal{A},\eta\in\mathcal{E}\otimes\Omega_D^1(\mathcal{A})\,. 
\end{eqnarray*}
A simple computation shows that $\,\varTheta:=\widetilde{\nabla}\circ\nabla$ is an element of $\mathcal{H}om_\mathcal{A}(\mathcal{E},\mathcal{E}
\otimes_\mathcal{A}\Omega_D^2(\mathcal{A}))$, i,e. it is $\mathcal{A}$-linear.

\begin{definition}
$\varTheta$ defined above is called the curvature of the connection $\nabla$.
\end{definition}

\textbf{The inner-product on} $\mathbf{\mathcal{H}om_\mathcal{A}(\mathcal{E},\mathcal{E}\otimes_\mathcal{A}\Omega_D^2):}$ First recall that
$\Omega_D^2\cong\pi(\Omega^2)/\pi(dJ_0^{(1)})$. Let $\mathcal{H}^\prime$ be the Hilbert space completion of $\pi(\Omega^2)$ with the inner-product 
\begin{eqnarray}\label{inner-product by Dixmier trace}
\langle T_1,T_2\rangle:= Tr_\omega(T_1^*T_2|D|^{-d}),\,\forall\,T_1,T_2\in\pi(\Omega^2) 
\end{eqnarray}
where $Tr_\omega$ denotes the Dixmier trace. Let $P$ be the orthogonal projection of $\mathcal{H}^\prime$ onto the orthogonal complement of the
subspace $\pi(dJ_0^{(1)})\subseteq\pi(\Omega^2)$. Now, define $\langle\,[T_1],[T_2]\,\rangle_{\Omega_D^2}:=\langle PT_1,PT_2\rangle,\,$ for all
$[T_j]\in\Omega_D^2$. This gives a well defined inner-product on $\Omega_D^2\,$. Viewing $\mathcal{E}=p\mathcal{A}^n$ we see that $\mathcal{H}om_\mathcal{A}
(\mathcal{E},\mathcal{E}\otimes_\mathcal{A}\Omega_D^2)=\mathcal{H}om_\mathcal{A}(p\mathcal{A}^n,p\mathcal{A}^n\otimes_\mathcal{A}\Omega_D^2)
\cong\mathcal{H}om_\mathcal{A}(pA^n,p(\Omega_D^2)^n)\,$, which is contained in $\mathcal{H}om_\mathcal{A}(\mathcal{A}^n,(\Omega_D^2)^n)$. Now,
for $\phi,\psi\in\mathcal{H}om_\mathcal{A}(\mathcal{E},\mathcal{E}\otimes_\mathcal{A}\Omega_D^2)$, define the inner-product as
\begin{eqnarray}\label{inner-product on curvature space}
\langle\langle\phi,\psi\rangle\rangle & := & \sum_{k=1}^n\sum_{i=1}^n\langle\,(\phi(e_k))_i\,,\,(\psi(e_k))_i\,\rangle_{\Omega_D^2}
\end{eqnarray}
where $\{e_1,\ldots,e_n\}$ is the standard canonical basis of the free module $\mathcal{A}^n$ over $\mathcal{A}$.

\begin{definition}\label{YM functional}
The Yang-Mills action functional is a map $\,\mathcal{YM}:C(\mathcal{E})\longrightarrow\mathbb{R}_{\geq 0}$ given by
\begin{center}
$\mathcal{YM}\,(\nabla)=\langle\langle\,\varTheta,\varTheta\,\rangle\rangle\,.$
\end{center}
\end{definition}

\begin{remark}\rm
The definition of $\mathcal{YM}$ does not depend on the choice of the projection used to describe $\mathcal E$. This is discussed in Remark
$[4.3]$ in (\cite{CG1}).
\end{remark}

Let $\nabla_t=\nabla+t\mu$ be a linear perturbation of a connection $\nabla$ on $\mathcal{E}$ by an element $\,\mu\in\mathcal{H}om_\mathcal{A}
(\mathcal{E}\,,\,\mathcal{E}\otimes_\mathcal{A}\Omega_D^1(\mathcal{A}))$. One can check that the curvature $\,\varTheta_t$ of the connection
$\nabla_t$ becomes $\varTheta+t[\nabla,\mu]+\mathcal{O}(t^2)$. If we suppose that $\nabla$ is an extremum of the Yang–Mills action functional,
this linear perturbation should not affect the action. In other words, we should have
\begin{center}
$\frac{d}{dt}|_{t=0}\,\mathcal{YM}(\nabla+t\mu)=0\,.$
\end{center}
From here it follows that $\langle\langle[\nabla,\mu],\varTheta\rangle\rangle=0$, where $\langle\langle\,\,,\,\rangle\rangle$ is the inner-product
on $\,Hom_\mathcal{A}(\mathcal{E},\mathcal{E}\otimes_\mathcal{A}\Omega_D^2)$ described in (\ref{inner-product on curvature space}). Here,
$[\nabla,\mu]=\widetilde{\nabla}\circ\mu+(1\otimes\Pi)\circ(\mu\otimes 1)\circ\nabla$, with $\,\Pi:\Omega_D^1\times\Omega_D^1\longrightarrow
\Omega_D^2$ being the product map of the Dirac dga. One can check that $[\nabla,\mu]$ is indeed $\mathcal{A}$-linear. Since $\mu$ is arbitrary,
we derive the equation of motion $[\nabla^*,\varTheta]=0$, where the adjoint of $\,[\nabla,.]$ is defined by
\begin{center}
$\langle\langle[\nabla^*,\alpha],\beta\rangle\rangle=\langle\langle\alpha,[\nabla,\beta]\rangle\rangle\,.$
\end{center}
For detail on these we refer (\cite{Lan}).

\begin{definition}\label{critical point}
A compatible connection $\nabla\in C(\mathcal{E})$ is called a critical point of Yang-Mills action functional if
\begin{center}
$\frac{d}{dt}|_{t=0}\,\mathcal{YM}(\nabla+t\mu)=0$
\end{center}
for all $\,\mu\in Hom_\mathcal{A}(\mathcal{E},\mathcal{E}\otimes_\mathcal{A}\Omega_D^1(\mathcal{A}))$.
\end{definition}
\bigskip


\section{Subadditivity and additivity of the Yang-Mills functional}

Let $(\mathcal{A}_1,\mathcal{H}_1,D_1,\gamma_1)$ and $(\mathcal{A}_2,\mathcal{H}_2,D_2,\gamma_2)$ be two even spectral triples. The product of
these, due to Connes (\cite{Con2}), is given by the following even spectral triple
\begin{eqnarray}\label{mult of spec}
&    & (\mathcal{A}_1,\mathcal{H}_1,D_1,\gamma_1)\otimes(\mathcal{A}_2,\mathcal{H}_2,D_2,\gamma_2)\nonumber\\
& := & (\mathcal{A}=\mathcal{A}_1\otimes \mathcal{A}_2\,,\,\mathcal{H}=\mathcal{H}_1\otimes\mathcal{H}_2\,,D=D_1\otimes 1+\gamma_1\otimes D_2
\,,\,\gamma=\gamma_1\otimes\gamma_2)\,.
\end{eqnarray}

At this point one should note the following.
\begin{enumerate}
\item One can also consider the following multiplication formula for even spectral triples
\begin{eqnarray*}
&   & (\mathcal{A}_1,\mathcal{H}_1,D_1,\gamma_1)\otimes(\mathcal{A}_2,\mathcal{H}_2,D_2,\gamma_2)\\
& = & (\mathcal{A}_1\otimes\mathcal{A}_2\,,\,\mathcal{H}_1\otimes\mathcal{H}_2\,,D^\prime=D_1\otimes\gamma_2+1\otimes D_2\,,\,\gamma_1\otimes\gamma_2)\,.
\end{eqnarray*}
In this case there exists a unitary $U\in\mathcal{B}(\mathcal{H}_1\otimes\mathcal{H}_2)$ given by
\begin{center}
$U:=\frac{1}{2}(1\otimes 1+\gamma_1\otimes 1+1\otimes\gamma_2-\gamma_1\otimes\gamma_2)$,
\end{center}
such that the spectral triples $(\mathcal{A}_1\otimes\mathcal{A}_2\,,\,\mathcal{H}_1\otimes\mathcal{H}_2\,,D)$ and $(\mathcal{A}_1\otimes
\mathcal{A}_2\,,\,\mathcal{H}_1\otimes\mathcal{H}_2\,,D^\prime)$ become unitary equivalent, i,e $\,UDU^*=D^\prime$ (\cite{Van}).
\item\label{from odd to even} If one starts with an {\it odd} spectral triple $(\mathcal{A},\mathcal{H},D)$, i,e. without the grading operator, then
one can construct an {\it even} spectral triple $(\mathcal{A},\widetilde{\mathcal{H}}=\mathcal{H}\otimes\mathbb{C}^2,\widetilde{D},\gamma)$ using
any two $2\times 2$ Pauli spin matrices such that $\,\Omega_D^\bullet(\mathcal{A})\cong\Omega_{\widetilde{D}}^\bullet(\mathcal{A})$ as dgas
(Lemma $2.7$ in \cite{CG3}). Therefore, when working with Dirac dga, one can w.l.o.g. assume that the spectral triple is always {\it even}.
\item If one of the spectral triple is {\it even} and the other is {\it odd} then the multiplication rule (\ref{mult of spec}) is still well-defined.
Only difference is that the resulting product spectral triple is now {\it odd}.
\item Let $\,\sigma_1,\sigma_2,\sigma_3$ be the $2\times 2$ Pauli spin matrices. For two {\it odd} spectral triples $(\mathcal{A}_j,\mathcal{H}_j,D_j)$,
$j=1,2$, one can also consider the following spectral triple
\begin{center}
$(\mathcal{A}_1\otimes\mathcal{A}_2\,,\,\mathcal{H}_1\otimes\mathcal{H}_2\otimes\mathbb{C}^2\,,D=D_1\otimes 1\otimes\sigma_1+1\otimes D_2\otimes\sigma_2)$ 
\end{center}
as their multiplication. It is an {\it even} spectral triple with the grading operator $1\otimes 1\otimes\sigma_3$. However, observe that this is
nothing but first making $(\mathcal{A}_1,\mathcal{H}_1,D_1)$ {\it even} as described above in point $(2)$, and then following the multiplication
rule given by (\ref{mult of spec}).
\end{enumerate}

We fix the following notations throughout the article.
\medskip

\textbf{Notation~:} $(a)\,\mathcal{A}=\mathcal{A}_1\otimes\mathcal{A}_2\,,\,(b)\,\mathcal{H}=\mathcal{H}_1\otimes\mathcal{H}_2\,,\,(c)\,D=D_1
\otimes 1+\gamma_1\otimes D_2\,,\,(d)\,\,\mathcal{E}=\mathcal{E}_1\otimes\mathcal{E}_2\,,\,(e)\,\pi=\pi_1\otimes\pi_2\,$.

\begin{lemma}\label{first form}
$\Omega_D^1(\mathcal{A})=\Omega_{D_1}^1(\mathcal{A}_1)\otimes\mathcal{A}_2\bigoplus\mathcal{A}_1\otimes\Omega_{D_2}^1(\mathcal{A}_2)$ as
$\mathcal{A}$-bimodules.
\end{lemma}
\begin{proof}
Note that $[D,\sum a_1\otimes a_2]=\sum [D_1,a_1]\otimes a_2+\gamma_1a_1\otimes[D_2,a_2]$. Since, $\gamma_1[D_1,a_1]=-[D_1,a_1]\gamma_1$
and $\gamma_1^2=1$, we have the following $\mathcal{A}$-bimodule inclusion,
\begin{center}
$\Omega_D^1(\mathcal{A})\subseteq\Omega_{D_1}^1(\mathcal{A}_1)\otimes\mathcal{A}_2\bigoplus\mathcal{A}_1\otimes\Omega_{D_2}^1(\mathcal{A}_2)\,.$
\end{center}
In order to show the equality, observe that any element $\sum a_0[D_1,a_1]\otimes a_2$ of $\Omega_{D_1}^1(\mathcal{A}_1)\otimes\mathcal{A}_2$
can be written as $\sum (a_0\otimes a_2)[D,a_1\otimes 1]$. Similarly, $\sum a_0\otimes a_1[D_2,a_2]\in\mathcal{A}_1\otimes\Omega_{D_2}^1
(\mathcal{A}_2)$ can be written as $\sum (a_0\otimes a_1)[D,1\otimes a_2]$.
\end{proof}

\begin{lemma}\label{first differential}
The Dirac dga differential $\,d:\mathcal{A}\longrightarrow\Omega_D^1(\mathcal{A})$ is given by
\begin{center}
$d(a_1\otimes a_2)=(d_1(a_1)\otimes a_2\,,\,a_1\otimes d_2(a_2))$
\end{center}
where $\,d_j:\mathcal{A}_j\longrightarrow\Omega_{D_j}^1(\mathcal{A}_j)$, for $j=1,2$, are the Dirac dga differentials associated with $\mathcal{A}_j$.
\end{lemma}
\begin{proof}
Follows from the previous Lemma (\ref{first form}). 
\end{proof}

\begin{lemma}\label{numerator of 2nd form}
$\pi(\Omega^2(\mathcal{A}))=\left(\pi_1(\Omega^2(\mathcal{A}_1))\otimes\mathcal{A}_2+\mathcal{A}_1\otimes\pi_2(\Omega^2(\mathcal{A}_2))\right)
\bigoplus\Omega_{D_1}^1(\mathcal{A}_1)\otimes\Omega_{D_2}^1(\mathcal{A}_2)$ as $\mathcal{A}$-bimodules.
\end{lemma}
\begin{proof}
Arbitrary element of $\,\pi(\Omega^2(\mathcal{A}))$ looks like
\begin{eqnarray*}
&  & \sum_{i,j,k}(a_{0i}\otimes a_{1i})[D,b_{0j}\otimes b_{1j}][D,c_{0k}\otimes c_{1k}]\\
& = & \sum_{i,j,k}a_{0i}[D_1,b_{0j}][D_1,c_{0k}]\otimes a_{1i}b_{1j}c_{1k}+a_{0i}b_{0j}c_{0k}\otimes a_{1i}[D_2,b_{1j}][D_2,c_{1k}]\\
&  & \quad\quad+\gamma_1\left(a_{0i}b_{0j}[D_1,c_{0k}]\otimes a_{1i}[D_2,b_{1j}]c_{1k}-a_{0i}[D_1,b_{0j}]c_{0k}\otimes a_{1i}b_{1j}[D_2,c_{1k}]\right)\,.
\end{eqnarray*}
Since, $\gamma_1$ anticommutes with $[D_1,a]$ but commutes with $\mathcal{A}_1$ and $[D_1,a][D_1,b]$ we have the inclusion `$\subseteq$'. Now,
\begin{itemize}
\item[(i)] $\underline{\Omega_{D_1}^1(\mathcal{A}_1)\otimes\Omega_{D_1}^1(\mathcal{A}_1)\subseteq\pi(\Omega^2(\mathcal{A}))\,:}\,\,$ Consider
$\sum_{i,j}a_{0i}[D_1,b_{0j}]\otimes a_{1i}[D_2,b_{1j}]\in\Omega_{D_1}^1(\mathcal{A}_1)\otimes\Omega_{D_1}^1(\mathcal{A}_1)$. Observe that
\begin{eqnarray*}
&  & \sum_{i,j}(a_{0i}\otimes a_{1i})[D,1\otimes b_{1j}][D,b_{0j}\otimes 1]=\sum_{i,j}\gamma_1a_{0i}[D_1,b_{0j}]\otimes a_{1i}[D_2,b_{1j}]\,.
\end{eqnarray*}
\item[(ii)] $\underline{\pi_1(\Omega^2(\mathcal{A}_1))\otimes\mathcal{A}_2\subseteq\pi(\Omega^2(\mathcal{A}))\,:}\,\,$ Consider $\sum_{i,j,k}
a_{0i}[D_1,b_{0j}][D_1,c_{0k}]\otimes a_{1i}\in\pi(\Omega^2(\mathcal{A}_1))\otimes\mathcal{A}_2$. Observe that
\begin{eqnarray*}
&  & \sum_{i,j,k}(a_{0i}\otimes a_{1i})[D,b_{0j}\otimes 1][D,c_{0k}\otimes 1]=\sum_{i,j,k}a_{0i}[D_1,b_{0j}][D_1,c_{0k}]\otimes a_{1i}\,.
\end{eqnarray*}
\item[(iii)] $\underline{\mathcal{A}_1\otimes\pi_2(\Omega^2(\mathcal{A}_2))\subseteq\pi(\Omega^2(\mathcal{A}))\,:}\,\,$ Take $\sum_{i,j,k}a_{0i}
\otimes a_{1i}[D_2,b_{1j}][D_2,c_{1k}]\in\mathcal{A}_1\otimes\pi(\Omega^2(\mathcal{A}_2))$. Observe that
\begin{eqnarray*}
&  & \sum_{i,j,k}(a_{0i}\otimes a_{1i})[D,1\otimes b_{1j}][D,1\otimes c_{1k}]=\sum_{i,j,k}a_{0i}\otimes a_{1i}[D_2,b_{1j}][D_2,c_{1k}]\,.
\end{eqnarray*}
\end{itemize}
This gives the reverse inclusion `$\supseteq$' and completes the proof.
\end{proof}

\begin{lemma}\label{denominator of 2nd form}
As an $\mathcal{A}$-bimodule, we have
\begin{center}
$\pi(dJ_0^1(\mathcal{A}))=\pi_1(d_1J_0^1(\mathcal{A}_1))\otimes\mathcal{A}_2+\mathcal{A}_1\otimes\pi_2(d_2J_0^1(\mathcal{A}_2))\,,$
\end{center}
i,e. $\pi(dJ_0^1(\mathcal{A}))\bigcap\left(\Omega_{D_1}^1(\mathcal{A}_1)\otimes\Omega_{D_2}^1(\mathcal{A}_2)\right)=\{0\}$ in $\pi(\Omega^2
(\mathcal{A}))$.
\end{lemma}
\begin{proof}
Arbitrary element of $\,\pi(dJ_0^1(\mathcal{A}))$ looks like
\begin{center}
$\sum_{j,k}[D,b_{0j}\otimes b_{1j}][D,c_{0k}\otimes c_{1k}]$
\end{center}
such that
\begin{eqnarray}\label{0th eqn}
\sum_{j,k}(b_{0j}\otimes b_{1j})[D,c_{0k}\otimes c_{1k}] & = & 0\,.
\end{eqnarray}
This equation (\ref{0th eqn}) gives us the following two equations (by Lemma [\ref{first form}])
\begin{eqnarray}\label{1st eqn}
\sum_{j,k}b_{0j}[D_1,c_{0k}]\otimes b_{1j}c_{1k} & = & 0\,,
\end{eqnarray}
\begin{eqnarray}\label{2nd eqn}
\sum_{j,k}b_{0j}c_{0k}\otimes b_{1j}[D_2,c_{1k}] & = & 0\,.
\end{eqnarray}
For arbitrary $\sum[D_1,b_{0j}][D_1,c_{0k}]\otimes a\in\pi_1(d_1J_0^1(\mathcal{A}_1))\otimes\mathcal{A}_2$ and $\sum b\otimes[D_2,b_{1j}]
[D_2,c_{1k}]\in\mathcal{A}_1\otimes\pi_2(d_2J_0^1(\mathcal{A}_2))$ if we denote $\xi$ to be their summation, then we see that $\xi=(1\otimes a)
\xi_1+(b\otimes 1)\xi_2$
where,
\begin{center}
$\xi_1=\sum[D,b_{0j}\otimes 1][D,c_{0k}\otimes 1]\quad$ and $\quad\xi_2=[D,1\otimes b_{1j}][D,1\otimes c_{1k}]\,.$
\end{center}
Both $\xi_1$ and $\xi_2$ are in $\pi(dJ_0^1(\mathcal{A}))$ by equations (\ref{1st eqn}) and (\ref{2nd eqn}). Being bimodule over $\mathcal{A}=
\mathcal{A}_1\otimes\mathcal{A}_2$ we conclude that $\xi\in\pi(dJ_0^1(\mathcal{A}))$. This proves the following,
\begin{center}
$\pi_1(d_1J_0^1(\mathcal{A}_1))\otimes\mathcal{A}_2+\mathcal{A}_1\otimes\pi_2(d_2J_0^1(\mathcal{A}_2))\subseteq\pi(dJ_0^1(\mathcal{A}))\,.$
\end{center}
To prove the reverse inclusion, first recall from  Lemma (\ref{first form}) that $J_0^1(\mathcal{A})=J_0^1(\mathcal{A}_1)\otimes\mathcal{A}_2
\bigoplus\mathcal{A}_1\otimes J_0^1(\mathcal{A}_2)$. Consider the element $\,\omega=\sum_{j,k}b_{0j}d_1(c_{0k})\otimes b_{1j}c_{1k}$ in
$J_0^1(\mathcal{A}_1)\otimes\mathcal{A}_2\subseteq J_0^1(\mathcal{A})$. So, $(1\otimes d_2)\omega\in J_0^1(\mathcal{A}_1)\otimes\Omega^1
(\mathcal{A}_2)$. Thus,
\begin{eqnarray}\label{3rd eqn}
&  & (\pi_1\otimes\pi_2)\circ(1\otimes d_2)\left(\sum_{j,k}b_{0j}d_1(c_{0k})\otimes b_{1j}c_{1k}\right)=0\nonumber\\
& \Rightarrow & \sum_{j,k}b_{0j}[D_1,c_{0k}]\otimes([D_2,b_{1j}]c_{1k}+b_{1j}[D_2,c_{1k}])=0\,.
\end{eqnarray}
Similarly, for $\,\omega=\sum_{j,k}b_{0j}c_{0k}\otimes b_{1j}d_2(c_{1k})$ in $\mathcal{A}_1\otimes J_0^1(\mathcal{A}_2)\subseteq J_0^1
(\mathcal{A})$ we get
\begin{eqnarray}\label{4th eqn}
&  & (\pi_1\otimes\pi_2)\circ(d_1\otimes 1)\left(\sum_{j,k}b_{0j}c_{0k}\otimes b_{1j}d_2(c_{1k})\right)=0\nonumber\\
& \Rightarrow & \sum_{j,k}([D_1,b_{0j}]c_{0k}+b_{0j}[D_1,c_{0k}])\otimes b_{1j}[D_2,c_{1k}]=0\,.
\end{eqnarray}
Finally, equation (\ref{3rd eqn})$-$(\ref{4th eqn}) gives us
\begin{eqnarray*}
\sum_{j,k}b_{0j}[D_1,c_{0k}]\otimes[D_2,b_{1j}]c_{1k}-[D_1,b_{0j}]c_{0k}\otimes b_{1j}[D_2,c_{1k}]=0\,.
\end{eqnarray*}
This implies that any arbitrary element $\,\xi=\sum_{j,k}[D,b_{0j}\otimes b_{1j}][D,c_{0k}\otimes c_{1k}]$ of $\pi(dJ_0^1(\mathcal{A}))$ is
actually of the form
\begin{center}
$\xi\,=\sum_{j,k}[D_1,b_{0j}][D_1,c_{0k}]\otimes b_{1j}c_{1k}+b_{0j}c_{0k}\otimes[D_2,b_{1j}][D_2,c_{1k}]\,.$
\end{center}
That is,
\begin{center}
$\pi(dJ_0^1(\mathcal{A}))\subseteq\pi_1(\Omega^2(\mathcal{A}_1))\otimes\mathcal{A}_2+\mathcal{A}_1\otimes\pi_2(\Omega^2(\mathcal{A}_2))\,,$
\end{center}
in view of Lemma (\ref{numerator of 2nd form}), i,e.  $\pi(dJ_0^1(\mathcal{A}))\bigcap\left(\Omega_{D_1}^1(\mathcal{A}_1)\otimes\Omega_{D_2}^1
(\mathcal{A}_2)\right)=\{0\}$ in $\pi(\Omega^2(\mathcal{A}))$. Now, recall a general result (Exercise 6, Part I, Chapter 2, Page 69 in \cite{Dx})
that given two Hilbert spaces $\mathcal{H}_1,\mathcal{H}_2$, and operators $T_i\in\mathcal{B}(\mathcal{H}_1),\,T_i^\prime\in\mathcal{B}
(\mathcal{H}_2),\,i=1,\ldots,k$, such that the $T_i$ are linearly independent and such that $\sum_{i=1}^k\,T_i\otimes T_i^\prime=0$, it follows that
$\,T_i^\prime=0$ for all $i=1,\ldots,k$. Using the faithfulness of $\,\pi_1,\pi_2$ in our case, we see that the equation (\ref{1st eqn}) implies
$\sum_{j,k}b_{0j}[D_1,c_{0k}]=0$ and equation (\ref{2nd eqn}) implies $\sum_{j,k}b_{1j}[D_2,c_{1k}]=0$. Hence,
\begin{center}
$\pi(dJ_0^1(\mathcal{A}))\subseteq\pi_1(d_1J_0^1(\mathcal{A}_1))\otimes\mathcal{A}_2+\mathcal{A}_1\otimes\pi_2(d_2J_0^1(\mathcal{A}_2))$
\end{center}
which concludes the proof.
\end{proof}

\begin{remark}\label{a counter example}\rm
In general, it is may not be true that
\begin{align*}
\pi_1(\Omega^2(\mathcal{A}_1))\otimes\mathcal{A}_2 &\bigcap \mathcal{A}_1\otimes\pi_2(\Omega^2(\mathcal{A}_2))=\{0\},\\
\pi_1(d_1J_0^1(\mathcal{A}_1))\otimes\mathcal{A}_2 &\bigcap \mathcal{A}_1\otimes\pi_2(d_2J_0^1(\mathcal{A}_2))=\{0\}.
\end{align*}
If we take $\mathcal{A}_1$ to be a noncommutative $n$-torus $\mathcal{A}_\Theta$ and $\mathcal{A}_2$ to be a noncommutative $m$-torus
$\mathcal{A}_\Phi$, then $\,\pi_1(\Omega^2(\mathcal{A}_\Theta))\otimes\mathcal{A}_\Phi\cong(\mathcal{A}_\Theta\otimes\mathcal{A}_\Phi)^{1+n(n-1)/2}$
and $\,\mathcal{A}_\Theta\otimes\pi_2(\Omega^2(\mathcal{A}_\Phi))\cong(\mathcal{A}_\Theta\otimes\mathcal{A}_\Phi)^{1+m(m-1)/2}$. But
\begin{eqnarray*}
\pi_1(\Omega^2(\mathcal{A}_\Theta))\otimes\mathcal{A}_\Phi+\mathcal{A}_\Theta\otimes\pi_2(\Omega^2(\mathcal{A}_\Phi))\cong
(\mathcal{A}_\Theta\otimes\mathcal{A}_\Phi)^{1+n(n-1)/2+m(m-1)/2}
\end{eqnarray*}
i,e. $\,\pi_1(\Omega^2(\mathcal{A}_\Theta))\otimes\mathcal{A}_\Phi\bigcap\mathcal{A}_\Theta\otimes\pi_2(\Omega^2(\mathcal{A}_\Phi))$ is nonzero,
and in fact can be identified with the free bimodule $\mathcal{A}_\Theta\otimes\mathcal{A}_\Phi$ of rank 1 over $\mathcal{A}_\Theta\otimes
\mathcal{A}_\Phi$. Moreover, $\pi(dJ_0^1(\mathcal{A}_\Theta\otimes\mathcal{A}_\Phi))\cong\mathcal{A}_\Theta\otimes\mathcal{A}_\Phi$ i,e. a free
bimodule over $\mathcal{A}_\Theta\otimes\mathcal{A}_\Phi$ of rank 1. But both $\pi_1(d_1J_0^1(\mathcal{A}_\Theta))\otimes\mathcal{A}_\Phi$ and
$\mathcal{A}_\Theta\otimes\pi_2(d_2J_0^1(\mathcal{A}_\Phi))$ are also free bimodules over $\mathcal{A}_\Theta\otimes\mathcal{A}_\Phi$ of rank 1.
Hence, by Lemma (\ref{denominator of 2nd form}) we see that $\pi_1(d_1J_0^1(\mathcal{A}_1))\otimes\mathcal{A}_2\bigcap\mathcal{A}_1\otimes\pi_2
(d_2J_0^1(\mathcal{A}_2))$ is nonzero, and in fact can be identified with a free bimodule over $\mathcal{A}_\Theta\otimes\mathcal{A}_\Phi$ of rank
1. We will see these in detail in Section $(4)$. This example explains that we can not replace the summation $`+$' in Lemma (\ref{numerator of 2nd form})
and (\ref{denominator of 2nd form}) by the direct sum $`\oplus$' always.
\end{remark}

Let $\mathcal{E}_1$ and $\mathcal{E}_2$ be two Hermitian f.g.p modules over $\mathcal{A}_1$ and $\mathcal{A}_2$ respectively. Then by Thm.
(\ref{structure thm of herm. mod.}) we know that $\mathcal{E}_1=p_1\mathcal{A}_1^m$ and $\mathcal{E}_2=p_2\mathcal{A}_2^n$ for suitable projections
$\,p_1\in M_m(\mathcal{A}_1),\,p_2\in M_n(\mathcal{A}_2)$ such that $\mathcal{E}_1,\,\mathcal{E}_2$ now have the canonical Hermitian structure on
them. Consider $\mathcal{E}:=\mathcal{E}_1\otimes\mathcal{E}_2$. Clearly, $\mathcal{E}$ is a f.g.p module over $\mathcal{A}$ of the form $\,p_1
\mathcal{A}_1^m\otimes p_2\mathcal{A}_2^n\cong(p_1\otimes p_2)(\mathcal{A}_1\otimes\mathcal{A}_2)^{mn}$.

\begin{lemma}\label{isomorphism involving product module}
For f.g.p right modules $\mathcal{E}_j$ over $\mathcal{A}_j,\,j=1,2$, one has
\begin{enumerate}
\item $(\mathcal{E}_1\otimes\mathcal{E}_2)\otimes_\mathcal{A}\left(\Omega_{D_1}^1(\mathcal{A}_1)\otimes\mathcal{A}_2\right)\cong
\left(\mathcal{E}_1\otimes_{\mathcal{A}_1}\Omega_{D_1}^1(\mathcal{A}_1)\right)\otimes\mathcal{E}_2$ as right $\mathcal{A}$-modules.
\item $(\mathcal{E}_1\otimes\mathcal{E}_2)\otimes_\mathcal{A}\left(\mathcal{A}_1\otimes\Omega_{D_2}^1(\mathcal{A}_2)\right)\cong
\mathcal{E}_1\otimes\left(\mathcal{E}_2\otimes_{\mathcal{A}_2}\Omega_{D_2}^1(\mathcal{A}_2)\right)$ as right $\mathcal{A}$-modules.
\end{enumerate}
\end{lemma}
\begin{proof}
These are canonical isomorphisms since both $\mathcal{A}_1,\mathcal{A}_2$ are unital algebras, and $\,\mathcal{E}_j\otimes_{\mathcal{A}_j}
\mathcal{A}_j\cong\mathcal{E}_j$ for $j=1,2$.
\end{proof}

\begin{lemma}
$\mathcal{E}:=\mathcal{E}_1\otimes\mathcal{E}_2$ is a Hermitian f.g.p module over $\mathcal{A}$ and the Hermitian structure is given by
\begin{center}
$\langle e_1\otimes e_2\,,\,e_1^\prime\otimes e_2^\prime\rangle_\mathcal{A}=\langle e_1\otimes e_1^\prime\rangle_{\mathcal{A}_1}\otimes\langle
e_2\otimes e_2^\prime\rangle_{\mathcal{A}_2}\,.$
\end{center}
\end{lemma}
\begin{proof}
Since, $\mathcal{E}$ is of the form $\,p_1\mathcal{A}_1^m\otimes p_2\mathcal{A}_2^n=(p_1\otimes p_2)(\mathcal{A}_1\otimes\mathcal{A}_2)^{mn}$,
where $p_1\otimes p_2\in M_{mn}(\mathcal{A})$ is a projection, restricting the canonical Hermitian structure on $\mathcal{A}^{mn}$ to $\mathcal{E}$
makes $\mathcal{E}$ a Hermitian f.g.p module over $\mathcal{A}$. One can easily verify the above equality.
\end{proof}

For two compatible connections $\nabla_1\in C(\mathcal{E}_1)$ and $\nabla_2\in C(\mathcal{E}_2)$ define
\begin{align*}
\nabla:\mathcal{E}_1\otimes\mathcal{E}_2 &\longrightarrow (\mathcal{E}_1\otimes\mathcal{E}_2)\otimes_\mathcal{A}\Omega_D^1(\mathcal{A})\\
e_1\otimes e_2 &\longmapsto \nabla_1(e_1)\otimes e_2+e_1\otimes\nabla_2(e_2)
\end{align*}

\begin{proposition}\label{the connection}
$\nabla\in C(\mathcal{E})$, i,e. if $\,\nabla_1$ and $\nabla_2$ are compatible connections on $\mathcal{E}_1$ and $\mathcal{E}_2$ respectively,
then so is $\nabla$ on $\mathcal{E}$.
\end{proposition}
\begin{proof}
Clearly, $\nabla$ is a $\mathbb{C}$-linear map. Now, for $\,e_1\otimes e_2\in\mathcal{E}$ and $x\otimes y\in\mathcal{A}$,
\begin{eqnarray*}
&  & \nabla((e_1\otimes e_2)(x\otimes y))\\
& = & \nabla(e_1x\otimes e_2y)\\
& = & \nabla_1(e_1x)\otimes e_2y+e_1x\otimes\nabla_2(e_2y)\\
& = & \nabla_1(e_1)x\otimes e_2y+(e_1\otimes d_1x)\otimes e_2y+e_1x\otimes\nabla_2(e_2)y+e_1x\otimes(e_2\otimes d_2y)\\
& = & (\nabla_1(e_1)x\otimes e_2y+e_1x\otimes\nabla_2(e_2)y)+(e_1\otimes e_2y\otimes d_1x\otimes 1+e_1x\otimes e_2\otimes d_2y)\\
& = & (\nabla_1(e_1)\otimes e_2+e_1\otimes\nabla_2(e_2))(x\otimes y)+(e_1\otimes e_2)\otimes(d_1x\otimes y+x\otimes d_2y)\\
& = & \nabla(e_1\otimes e_2)(x\otimes y)+(e_1\otimes e_2)\otimes d(x\otimes y)
\end{eqnarray*}
by Lemma (\ref{first differential}). Hence, $\nabla$ is a connection on $\mathcal{E}$. Now, we show that $\nabla$ is compatible with
respect to the Hermitian structure on $\mathcal{E}$. Let
\begin{center}
 $\nabla_1(e_1)=\sum_i e_{1i}\otimes\omega_{1i}\,\,\in\mathcal{E}_1\otimes\Omega_{D_1}^1(\mathcal{A}_1)\,,$
\end{center}
\begin{center}
 $\nabla_2(e_2)=\sum_i e_{2i}\otimes\omega_{2i}\,\,\in\mathcal{E}_2\otimes\Omega_{D_2}^1(\mathcal{A}_2)\,,$
\end{center}
\begin{center}
 $\nabla_1(e_1^\prime)=\sum_i e_{1i}^\prime\otimes\omega_{1i}^\prime\in\mathcal{E}_1\otimes\Omega_{D_1}^1(\mathcal{A}_1)\,,$
\end{center}
\begin{center}
 $\nabla_2(e_2^\prime)=\sum_i e_{2i}^\prime\otimes\omega_{2i}^\prime\in\mathcal{E}_2\otimes\Omega_{D_2}^1(\mathcal{A}_2)\,$.
\end{center}
Then,
\begin{align*}
\nabla_1(e_1)\otimes e_2+e_1\otimes\nabla_2(e_2) &= \sum_i (e_{1i}\otimes e_2\otimes\omega_{1i}\,,\,e_1\otimes e_{2i}\otimes\omega_{2i})\,,\\
\nabla_1(e_1^\prime)\otimes e_2^\prime+e_1^\prime\otimes\nabla_2(e_2^\prime) &= \sum_i (e_{1i}^\prime\otimes e_2^\prime\otimes\omega_{1i}^\prime
 \,,\,e_1^\prime\otimes e_{2i}^\prime\otimes\omega_{2i}^\prime)\,.
\end{align*}
and
\begin{eqnarray}\label{eqn involving Dirac dgas}
 d\langle e_1\otimes e_2\,,\,e_1^\prime\otimes e_2^\prime\rangle & = & d(\langle e_1,e_1^\prime\rangle\otimes\langle e_2,e_2^\prime\rangle)\\
 & = & d_1(\langle e_1,e_1^\prime\rangle)\otimes\langle e_2,e_2^\prime\rangle+\langle e_1,e_1^\prime\rangle\otimes d_2(\langle e_2,e_2^\prime\rangle)\,.\nonumber
\end{eqnarray}
Since, $\nabla_1\in C(\mathcal{E}_1)$ and $\nabla_2\in C(\mathcal{E}_2)$ we have
\begin{eqnarray*}
&  & \langle e_1,\nabla_1e_1^\prime\rangle-\langle\nabla_1e_1,e_1^\prime\rangle=d_1(\langle e_1,e_1^\prime\rangle)\\
&  & \langle e_2,\nabla_2e_2^\prime\rangle-\langle\nabla_2e_2,e_2^\prime\rangle=d_2(\langle e_2,e_2^\prime\rangle)
\end{eqnarray*}
which further implies the following equations
\begin{eqnarray}\label{individual compatibility 1}
\sum_i\langle e_1,e_{1i}^\prime\rangle\omega_{1i}^\prime-\omega_{1i}^*\langle e_{1i},e_1^\prime\rangle & = & d_1(\langle e_1,e_1^\prime\rangle)\,,
\end{eqnarray}
\begin{eqnarray}\label{individual compatibility 2}
\sum_i\langle e_2,e_{2i}^\prime\rangle\omega_{2i}^\prime-\omega_{2i}^*\langle e_{2i},e_2^\prime\rangle & = & d_2(\langle e_2,e_2^\prime\rangle)\,.
\end{eqnarray}
Now,
\begin{eqnarray*}
&  & \langle e_1\otimes e_2\,,\,\nabla(e_1^\prime\otimes e_2^\prime)\rangle\\
& = & \langle e_1\otimes e_2\,,\,\sum_ie_{1i}^\prime\otimes e_2^\prime\otimes\omega_{1i}^\prime+e_1^\prime\otimes e_{2i}^\prime\otimes\omega_{2i}^\prime\rangle\\
& = & \sum_{i,j}\langle e_1\otimes e_2\,,\,e_{1i}^\prime\otimes e_2^\prime\otimes a_{01ij}^\prime[D_1,a_{11ij}^\prime]+e_1^\prime\otimes e_{2i}^\prime
\otimes a_{02ij}^\prime[D_2,a_{12ij}^\prime]\rangle\\
& = & \sum_{i,j}\langle e_1\otimes e_2\,,\,e_{1i}^\prime\otimes e_2^\prime\otimes((a_{01ij}^\prime\otimes 1)[D,a_{11ij}^\prime\otimes 1])+e_1^\prime
\otimes e_{2i}^\prime\otimes((1\otimes a_{02ij}^\prime)[D,1\otimes a_{12ij}^\prime])\rangle\\
& = & \sum_{i,j}(\langle e_1,e_{1i}^\prime\rangle\otimes\langle e_2,e_2^\prime\rangle)((a_{01ij}^\prime\otimes 1)[D,a_{11ij}^\prime\otimes 1])\\
&  & +\sum_{i,j}(\langle e_1,e_1^\prime\rangle\otimes\langle e_2,e_{2i}^\prime\rangle)((1\otimes a_{02ij}^\prime)[D,1\otimes a_{12ij}^\prime])\\
& = & \sum_{i,j}(\langle e_1,e_{1i}^\prime\rangle a_{01ij}^\prime\otimes\langle e_2,e_2^\prime\rangle)[D,a_{11ij}^\prime\otimes 1]
+(\langle e_1,e_1^\prime\rangle\otimes\langle e_2,e_{2i}^\prime\rangle a_{02ij}^\prime)[D,1\otimes a_{12ij}^\prime]\\
& = & \sum_i\langle e_1,e_{1i}^\prime\rangle\omega_{1i}^\prime\otimes\langle e_2,e_2^\prime\rangle+\langle e_1,e_1^\prime\rangle\otimes
\langle e_2,e_{2i}^\prime\rangle\omega_{2i}^\prime
\end{eqnarray*}
Similarly,
\begin{eqnarray*}
&  & \langle\nabla(e_1\otimes e_2)\,,\,e_1^\prime\otimes e_2^\prime\rangle=\sum_i\omega_{1i}^*\langle e_{1i},e_1^\prime\rangle\otimes\langle e_2,
e_2^\prime\rangle+\langle e_1,e_1^\prime\rangle\otimes\omega_{2i}^*\langle e_{2i},e_2^\prime\rangle\,.
\end{eqnarray*}
Subtracting, we get from equations (\ref{individual compatibility 1}), (\ref{individual compatibility 2}) and (\ref{eqn involving Dirac dgas}) that
\begin{eqnarray*}
&  & \langle e_1\otimes e_2\,,\,\nabla(e_1^\prime\otimes e_2^\prime)\rangle-\langle\nabla(e_1\otimes e_2)\,,\,e_1^\prime\otimes e_2^\prime\rangle=
d(\langle e_1\otimes e_2\,,\,e_1^\prime\otimes e_2^\prime\rangle)
\end{eqnarray*}
This proves that $\nabla$ is a compatible connection i,e. $\nabla\in C(\mathcal{E})$.
\end{proof}

\begin{remark}\rm
Individually, $\nabla_1\otimes 1$ and $1\otimes\nabla_2$ are not connections on $\mathcal{E}=\mathcal{E}_1\otimes\mathcal{E}_2$.
\end{remark}

Now, we are in a position to define subadditivity and additivity of the Yang-Mills functional.

\begin{definition}\label{defn of additivity and subadditivity}
For any two even spectral triples $(\mathcal{A}_j,\mathcal{H}_j,D_j,\gamma_j)$ and Hermitian f.g.p modules $\,\mathcal{E}_j$ over $\mathcal{A}_j,\,j=1,2$,
we say that the Yang-Mills action functional $\mathcal{YM}$ is
\begin{enumerate}
\item Subadditive if $\,\sqrt{\mathcal{YM}}(\nabla)\leq\sqrt{\alpha}\sqrt{\mathcal{YM}}(\nabla_1)+\sqrt{\beta}\sqrt{\mathcal{YM}}(\nabla_2)\,$,
for all $\nabla_j\in C(\mathcal{E}_j)\,,$
\item Additive if $\,\mathcal{YM}(\nabla)=\alpha \mathcal{YM}(\nabla_1)+\beta \mathcal{YM}(\nabla_2)\,$, for all $\nabla_j\in C(\mathcal{E}_j)\,;$
\end{enumerate}
for certain positive constants $\alpha$ and $\beta$, essentially determined by the summability of the individual spectral triples
$($These constants will be explicitly determined in Thm. $[\ref{subadditivity}])$.
\end{definition}

\begin{remark}\rm
Above definition (\ref{defn of additivity and subadditivity}) is natural in the following sense. The Yang-Mills action functional is defined using
certain inner-product (Def. [\ref{YM functional}]). Hence, the square root is given by a certain norm. However, one should note that the domain of
the Yang-Mills functional is an affine space instead of a vector space. Therefore, a suitable formulation of subadditivity and additivity was needed.
\end{remark}

Now, we put an assumption on the individual spectral triples $(\mathcal{A}_j,\mathcal{H}_j,D_j),\,j=1,2$, to show that the Yang-Mills action
functional is always subadditive.
\medskip

\textbf{Assumption:} 
For $\mathcal{A}=\mathcal{A}_1\otimes\mathcal{A}_2\,$,
\begin{center}
$\frac{\pi_1(\Omega^2(\mathcal{A}_1))\otimes\mathcal{A}_2+\mathcal{A}_1\otimes\pi_2(\Omega^2(\mathcal{A}_2))}{\pi_1(d_1J_0^1(\mathcal{A}_1))
\otimes\mathcal{A}_2+\mathcal{A}_1\otimes\pi_2(d_2J_0^1(\mathcal{A}_2))}\cong\Omega_{D_1}^2(\mathcal{A}_1)\otimes\mathcal{A}_2\bigoplus\mathcal{A}_1
\otimes\Omega_{D_2}^2(\mathcal{A}_2)$ 
\end{center} as $\mathcal{A}$-bimodules.
\medskip

\begin{lemma}\label{2 form for product system}
The above assumption is equivalent to the fact that
\begin{center}
$\Omega_D^2(\mathcal{A})\cong\Omega_{D_1}^2(\mathcal{A}_1)\otimes\mathcal{A}_2\bigoplus\mathcal{A}_1\otimes\Omega_{D_2}^2(\mathcal{A}_2)\bigoplus
\Omega_{D_1}^1(\mathcal{A}_1)\otimes\Omega_{D_2}^1(\mathcal{A}_2)$
\end{center} as $\mathcal{A}$-bimodules.
\end{lemma}
\begin{proof}
Follows from Lemma (\ref{numerator of 2nd form}) and (\ref{denominator of 2nd form}).
\end{proof}

In general, it is not known whether Lemma $(\ref{2 form for product system})$ is always true for any pair of spectral triples $(\mathcal{A}_j,
\mathcal{H}_j,D_j),\,j=1,2$. One has to check this for each particular cases. After the end of this section we provide few examples to validate
this assumption.

\begin{lemma}\label{2nd differential}
The Dirac dga differential $\,d:\Omega_D^1(\mathcal{A})\longrightarrow\Omega_D^2(\mathcal{A})$ is given by the following
\begin{center}
$d(\omega_1\otimes a_2\,,\,a_1\otimes\omega_2)=\left(d_1(\omega_1)\otimes a_2\,,\,a_1\otimes d_2(\omega_2)\,,\,\omega_1\otimes d_2(a_2)-d_1(a_1)
\otimes\omega_2\right)$
\end{center}
where, for $j=1,2,\,\,d_j:\Omega_{D_j}^1(\mathcal{A}_j)\longrightarrow\Omega_{D_j}^2(\mathcal{A}_j)$ are the Dirac dga differentials
associated with $\mathcal{A}_j$.
\end{lemma}
\begin{proof}
Let $\,\omega_1=\sum a_{0i}[D_1,a_{1i}]$ and $\,\omega_2=\sum b_{0i}[D_2,b_{1i}]$. Then,
\begin{eqnarray*}
(\omega_1\otimes a_2\,,\,a_1\otimes\omega_2) & = & \sum_i\left(a_{0i}[D_1,a_{1i}]\otimes a_2\,,\,a_1\otimes b_{0i}[D_2,b_{1i}]\right)\\
& = & \sum_i(a_{0i}\otimes a_2)[D,a_{1i}\otimes 1]+(a_1\otimes b_{0i})[D,1\otimes b_{1i}]
\end{eqnarray*}
as an element of $\Omega_D^1(\mathcal{A})$ (Lemma [\ref{first form}]). Hence,
\begin{eqnarray*}
d(\omega_1\otimes a_2\,,\,a_1\otimes\omega_2) & = & \sum_i[D,a_{0i}\otimes a_2][D,a_{1i}\otimes 1]+[D,a_1\otimes b_{0i}][D,1\otimes b_{1i}]\\
& = & \sum_i [D_1,a_{0i}][D_1,a_{1i}]\otimes a_2+a_1\otimes[D_2,b_{0i}][D_2,b_{1i}]\\
&  & \quad\,\,+\gamma_1a_{0i}[D_1,a_{1i}]\otimes[D_2,a_2]-\gamma_1[D_1,a_1]\otimes b_{0i}[D_2,b_{1i}]\\
& = & d_1(\omega_1)\otimes a_2+a_1\otimes d_2(\omega_2)+\gamma_1(\omega_1\otimes d_2(a_2)-d_1(a_1)\otimes\omega_2).
\end{eqnarray*}
Our conclusion now follows from the previous Lemma (\ref{2 form for product system}).
\end{proof}

\begin{lemma}\label{product map}
The product map $\,\Pi:\Omega_D^1(\mathcal{A})\times\Omega_D^1(\mathcal{A})\longrightarrow\Omega_D^2(\mathcal{A})$ is given by the following
\begin{itemize}
\item[(i)] $(\omega_1\otimes a_2\,,\,0).(\omega_1^\prime\otimes a_2^\prime\,,\,0)=(\omega_1\omega_1^\prime\otimes a_2a_2^\prime\,,\,0\,,\,0)$.
\item[(ii)] $(0\,,\,a_1\otimes\omega_2).(0\,,\,a_1^\prime\otimes\omega_2^\prime)=(0\,,\,a_1a_1^\prime\otimes\omega_2\omega_2^\prime\,,\,0)$.
\item[(iii)] $(0\,,\,a_1\otimes\omega_2).(\omega_1^\prime\otimes a_2^\prime\,,\,0)=(0\,,\,0\,,\,a_1\omega_1^\prime\otimes\omega_2a_2^\prime)$.
\item[(iv)] $(\omega_1\otimes a_2\,,\,0).(0\,,\,a_1^\prime\otimes\omega_2^\prime)=(0\,,\,0\,,\,-\omega_1a_1^\prime\otimes a_2\omega_2^\prime)$.
\end{itemize}
\end{lemma}
\begin{proof}
Part $(i),(ii),(iii)$ are straightforward verification using Lemma (\ref{first form}\,,\,\ref{numerator of 2nd form}\,,\,\ref{2 form for product system}).
We only explain part $(iv)$ to show why the minus sign appears. Let $\,\omega_1=\sum_ix_{0i}[D_1,x_{1i}]$ and $\,\omega_2^\prime=\sum_iy_{0i}[D_2,y_{1i}]$.
Then $\,\omega_1\otimes a_2=\sum_i(x_{0i}\otimes a_2)[D,x_{1i}\otimes 1]$ and $a_1^\prime\otimes\omega_2^\prime=\sum_i(a_1^\prime\otimes y_{0i})
[D,1\otimes y_{1i}]$ as elements of $\Omega_D^1(\mathcal{A})$. Hence,
\begin{eqnarray*}
(\omega_1\otimes a_2\,,\,0).(0\,,\,a_1^\prime\otimes\omega_2^\prime)
& = & \sum_{i,j}(x_{0i}\otimes a_2)[D,x_{1i}\otimes 1](a_1^\prime\otimes y_{0j})[D,1\otimes y_{1j}]\\
& = & \sum_{i,j}(x_{0i}[D_1,x_{1i}]\otimes a_2)(\gamma_1a_1^\prime\otimes y_{0j}[D_2,y_{1j}])\\
& = & \sum_{i,j}-\gamma_1x_{0i}[D_1,x_{1i}]a_1^\prime\otimes a_2y_{0j}[D_2,y_{1j}]\\
& = & -\gamma_1\omega_1a_1^\prime\otimes a_2\omega_2^\prime
\end{eqnarray*}
This element is identified with $(0\,,\,0\,,\,-\omega_1a_1^\prime\otimes a_2\omega_2^\prime)\in\Omega_D^2(\mathcal{A})$ by Lemma (\ref{2 form for product system}).
\end{proof}

\begin{proposition}\label{the curvature}
The curvature $\,\varTheta$ of the connection $\nabla$ is given by
\begin{center}
$\varTheta(e_1\otimes e_2)=\varTheta_1(e_1)\otimes e_2+e_1\otimes\varTheta_2(e_2)$
\end{center}
where, $\,\varTheta_1,\varTheta_2$ are the curvatures associated to the connections $\nabla_1,\nabla_2$ respectively.
\end{proposition}
\begin{proof}
For $\nabla_1(e_1)=\sum_ix_i\otimes\omega_i$ and $\nabla_2(e_2)=\sum_i y_i\otimes v_i$, we have
\begin{eqnarray*}
\nabla(e_1\otimes e_2) & = & \nabla_1(e_1)\otimes e_2+e_1\otimes\nabla_2(e_2)\\
& = & \sum_ix_i\otimes\omega_i\otimes e_2+e_1\otimes y_i\otimes v_i\\
& = & \sum_i(x_i\otimes e_2)\otimes\omega_i+(e_1\otimes y_i)\otimes v_i
\end{eqnarray*}
using Lemma (\ref{isomorphism involving product module}). Since $\varTheta=\widetilde{\nabla}\circ\nabla$, we get using Lemma (\ref{2nd differential})
\begin{eqnarray*}
\widetilde{\nabla}\circ\nabla(e_1\otimes e_2) & = & \widetilde{\nabla}(\nabla_1(e_1)\otimes e_2+e_1\otimes\nabla_2(e_2))\\
& = & \sum_i\nabla(x_i\otimes e_2)\omega_i+x_i\otimes e_2\otimes d\omega_i+\nabla(e_1\otimes y_i)v_i+e_1\otimes y_i\otimes dv_i\\
& = & \sum_i(\nabla_1(x_i)\otimes e_2)\omega_i+(x_i\otimes \nabla_2(e_2))\omega_i+x_i\otimes e_2\otimes d\omega_i+e_1\otimes y_i\otimes dv_i\\
&   &\quad\,\,+(\nabla_1(e_1)\otimes y_i)v_i+(e_1\otimes\nabla_2(y_i))v_i\\
& = & \sum_i\left(\nabla_1(x_i)\omega_i\otimes e_2+(x_i\otimes d_1\omega_i)\otimes e_2+e_1\otimes\nabla_2(y_i)v_i+e_1\otimes y_i\otimes d_2v_i\right)\\
&   &\quad\quad+\left(\sum_jx_i\otimes y_j\otimes v_j\omega_i+x_j\otimes y_i\otimes\omega_jv_i\right)\\
& = & \sum_i\widetilde{\nabla_1}(x_i\otimes\omega_i)\otimes e_2+e_1\otimes\widetilde{\nabla_2}(y_i\otimes v_i)\\
& = & \widetilde{\nabla_1}(\nabla_1(e_1))\otimes e_2+e_1\otimes\widetilde{\nabla_1}(\nabla_2(e_2))\\
& = & \varTheta_1(e_1)\otimes e_2+e_1\otimes\varTheta_2(e_2)\,.
\end{eqnarray*}
The third equality from below comes from part $(iii)$ and $(iv)$ of Lemma (\ref{product map}).
\end{proof}

\begin{lemma}\label{component of curvature are linear}
Both $\varTheta_1\otimes 1$ and $1\otimes\varTheta_2$ are in $\mathcal{H}om_\mathcal{A}(\mathcal{E},\mathcal{E}\otimes_\mathcal{A}\Omega_D^2
(\mathcal{A}))$.
\end{lemma}
\begin{proof}
It is easy to verify that both $\varTheta_1\otimes 1$ and $1\otimes\varTheta_2$ are $\mathcal{A}$-linear, because $\varTheta_j$ are
$\mathcal{A}_j$-linear for $j=1,2$. Conclusion now follows from Lemma (\ref{2 form for product system}).
\end{proof}

Recall the inner-product on $\pi(\Omega^2(\mathcal{A}))$ from (\ref{inner-product by Dixmier trace}). By Lemma (\ref{numerator of 2nd form}),
we have the induced inner-product on the subspaces $\,\pi_1(\Omega^2(\mathcal{A}_1))\otimes\mathcal{A}_2\,,\,\mathcal{A}_1\otimes\pi_2
(\Omega^2(\mathcal{A}_2))$ and $\Omega_{D_1}^1(\mathcal{A}_1)\otimes\Omega_{D_2}^1(\mathcal{A}_2)$ of $\pi(\Omega^2(\mathcal{A}))$.

\begin{lemma}\label{induced inner-product}
The induced inner-product on the subspaces $\,\pi_1(\Omega^2(\mathcal{A}_1))\otimes\mathcal{A}_2\,,\,\mathcal{A}_1\otimes\pi_2(\Omega^2
(\mathcal{A}_2))$ and $\Omega_{D_1}^1(\mathcal{A}_1)\otimes\Omega_{D_2}^1(\mathcal{A}_2)$ are given, upto multiplication by a positive constant,
by the following
\begin{itemize}
\item[(i)] $\langle\,T_1\otimes a_2\,,\,T_1^\prime\otimes a_2^\prime\rangle=Tr_\omega(T_1^*T_1^\prime|D_1|^{-k})Tr_\omega(a_2^*a_2^\prime|D_2|^{-l})$
\item[(ii)] $\langle\,a_1\otimes T_2\,,\,a_1^\prime\otimes T_2^\prime\rangle=Tr_\omega(a_1^*a_1^\prime|D_1|^{-k})Tr_\omega(T_2^*T_2^\prime|D_2|^{-l})$
\item[(iii)] $\langle\,T_1\otimes S_1\,,\,T_1^\prime\otimes S_1^\prime\rangle=Tr_\omega(T_1^*T_1^\prime|D_1|^{-k})Tr_\omega(S_1^*S_1^\prime|D_2|^{-l})$
\end{itemize}
respectively, where $Tr_\omega$ denotes the Dixmier trace.
\end{lemma}
\begin{proof}
Assume that $(\mathcal{A}_1,\mathcal{H}_1,D_1)$ is a $k$-summable spectral triple and $(\mathcal{A}_2,\mathcal{H}_2,D_2)$ is a $\ell$-summable
spectral triple. Then, $(\mathcal{A},\mathcal{H},D)$ is a $(k+\ell)$-summable spectral triple, and we have
\begin{eqnarray*}
\frac{\Gamma\left(\frac{k+\ell}{2}+1\right)}{\Gamma(\frac{k}{2}+1)\Gamma(\frac{\ell}{2}+1)}Tr_\omega\left((T_1\otimes T_2)|D|^{-(k+\ell)}\right)
& = & Tr_\omega\left(T_1|D_1|^{-k}\right)Tr_\omega\left(T_2|D_2|^{-\ell}\right)
\end{eqnarray*}
for all $T_j\in\mathcal{B}(\mathcal{H}_j),\,j=1,2$, (Page $576$ in \cite{Con2}) where $Tr_\omega$ denotes the Dixmier trace. The number
$\frac{\Gamma\left(\frac{k+\ell}{2}+1\right)}{\Gamma(\frac{k}{2}+1)\Gamma(\frac{\ell}{2}+1)}$ is a positive real constant, and this completes
the proof.
\end{proof}

\begin{proposition}\label{an automatic orthogonality}
The subspaces $\,\pi_1(\Omega^2(\mathcal{A}_1))\otimes\mathcal{A}_2+\mathcal{A}_1\otimes\pi_2(\Omega^2(\mathcal{A}_2))$ and $\,\Omega_{D_1}^1(
\mathcal{A}_1)\otimes\Omega_{D_2}^1(\mathcal{A}_2)$ of $\,\pi(\Omega^2(\mathcal{A}))$ are orthogonal to each other.
\end{proposition}
\begin{proof}
Recall from Lemma (\ref{numerator of 2nd form}) that arbitrary element of $\Omega_{D_1}^1(\mathcal{A}_1)\otimes\Omega_{D_2}^1(\mathcal{A}_2)\subseteq
\pi(\Omega^2(\mathcal{A}))$ is of the form $\sum_{i,j}(a_{0i}\otimes a_{1i})[D,1\otimes b_{1j}][D,b_{0j}\otimes 1]=\sum_{i,j}\gamma_1a_{0i}[D_1,b_{0j}]
\otimes a_{1i}[D_2,b_{1j}]$. Let $\xi$ be this element. For any element $\,\eta=\eta_1\otimes\eta_2\in\pi_1(\Omega^2(\mathcal{A}_1))\otimes\mathcal{A}_2+
\mathcal{A}_1\otimes\pi_2(\Omega^2(\mathcal{A}_2))\subseteq\pi(\Omega^2(\mathcal{A}))$ it follows, in view of the inner-product given in
(\ref{inner-product by Dixmier trace}) and the fact
\begin{eqnarray*}
\frac{\Gamma\left(\frac{k+\ell}{2}+1\right)}{\Gamma(\frac{k}{2}+1)\Gamma(\frac{\ell}{2}+1)}Tr_\omega\left((T_1\otimes T_2)|D|^{-(k+\ell)}\right)
& = & Tr_\omega\left(T_1|D_1|^{-k}\right)Tr_\omega\left(T_2|D_2|^{-\ell}\right)\,,
\end{eqnarray*}
that
\begin{eqnarray*}
&  & \langle\xi,\eta\rangle\\
& = & Tr_\omega\left(\xi^*\eta|D|^{-(k+\ell)}\right)\\
& = & \sum_{i,j}\left(\frac{\Gamma\left(\frac{k+\ell}{2}+1\right)}{\Gamma(\frac{k}{2}+1)\Gamma(\frac{\ell}{2}+1)}\right)^{-1}Tr_\omega\left([D_1,b_{0j}]^*
a_{0i}^*\gamma_1\eta_1|D_1|^{-k}\right)\,Tr_\omega\left([D_2,b_{1j}]^*a_{1i}\eta_2|D_2|^{-\ell}\right)\\
& = & \sum_{i,j}\left(\frac{\Gamma\left(\frac{k+\ell}{2}+1\right)}{\Gamma(\frac{k}{2}+1)\Gamma(\frac{\ell}{2}+1)}\right)^{-1}Tr_\omega\left([D_1,b_{0j}]^*
a_{0i}^*\eta_1|D_1|^{-k}\gamma_1\right)\,Tr_\omega\left([D_2,b_{1j}]^*a_{1i}\eta_2|D_2|^{-\ell}\right)\\
& = & \sum_{i,j}\left(\frac{\Gamma\left(\frac{k+\ell}{2}+1\right)}{\Gamma(\frac{k}{2}+1)\Gamma(\frac{\ell}{2}+1)}\right)^{-1}Tr_\omega\left(\gamma_1[D_1,b_{0j}]^*
a_{0i}^*\eta_1|D_1|^{-k}\right)\,Tr_\omega\left([D_2,b_{1j}]^*a_{1i}\eta_2|D_2|^{-\ell}\right)\\
& = & -\sum_{i,j}\left(\frac{\Gamma\left(\frac{k+\ell}{2}+1\right)}{\Gamma(\frac{k}{2}+1)\Gamma(\frac{\ell}{2}+1)}\right)^{-1}Tr_\omega\left([D_1,b_{0j}]^*a_{0i}^*
\gamma_1\eta_1|D_1|^{-k}\right)\,Tr_\omega\left([D_2,b_{1j}]^*a_{1i}\eta_2|D_2|^{-\ell}\right)\\
& = & -\langle\xi,\eta\rangle
\end{eqnarray*}
Here, the minus sign appears at the end because of the facts that $\,Tr_\omega$ is a trace, $\gamma_1$ commutes with $\eta_1$ and $a_{0i}$,
anticommutes with $D_1$, hence commutes with $|D_1|$ and $|D_1|^{-k}$, anticommutes with $[D_1,b_{0j}]$. Since, $\,\xi$ and $\eta$ are arbitrary,
our claim follows.
\end{proof}

Conclusion of the above Propn. (\ref{an automatic orthogonality}) is that the algebraic direct sum in Lemma (\ref{numerator of 2nd form}) is an
orthogonal direct sum with the respect to the inner-product (\ref{inner-product by Dixmier trace}) on $\pi(\Omega^2(\mathcal{A}))$. That is, the
second algebraic direct sum in Lemma (\ref{2 form for product system}) is always an orthogonal direct sum. However, in general, the first algebraic
direct sum in Lemma (\ref{2 form for product system}) fails to be an orthogonal direct sum.

\begin{theorem}\label{subadditivity}
The Yang-Mills action functional is always subadditive.
\end{theorem}
\begin{proof}
Let $\mathcal{E}_1=p_1\mathcal{A}_1^m$ and $\mathcal{E}_2=p_2\mathcal{A}_2^n$ be Hermitian f.g.p modules over $\mathcal{A}_1$ and $\mathcal{A}_2$
respectively, where the Hermitian structures are the induced canonical structure. Then, $\,\mathcal{E}=\mathcal{E}_1\otimes\mathcal{E}_2$ is a
Hermitian f.g.p module over $\mathcal{A}$ where, the Hermitian structure is the induced canonical structure. Moreover, $\mathcal{E}$ is $\,p_1
\mathcal{A}_1^m\otimes p_2\mathcal{A}_2^n\cong(p_1\otimes p_2)\mathcal{A}^{mn}$. Let $\{\sigma_1,\ldots,\sigma_{m}\}$ be the standard basis of
$\mathcal{A}_1^m$ as free module over $\mathcal{A}_1$ and $\{\mu_1,\ldots,\mu_{n}\}$ be that of $\mathcal{A}_2^n$. Then, $\{\sigma_i\otimes\mu_j\}$
is the standard basis of $\mathcal{A}^{mn}$ as free module over $\mathcal{A}$. We assume $(\mathcal{A}_1,\mathcal{H}_1,D_1)$ is a $k$-summable spectral
triple and $(\mathcal{A}_2,\mathcal{H}_2,D_2)$ is a $\ell$-summable spectral triple. Then $(\mathcal{A},\mathcal{H},D)$ is a $(k+\ell)$-summable
spectral triple, and recall from Propn. (\ref{the curvature}) that $\varTheta=\varTheta_1\otimes 1+1\otimes\varTheta_2\,$. Since, both $\varTheta_1
\otimes 1$ and $1\otimes\varTheta_2$ are $\mathcal{A}$-linear maps (Lemma [\ref{component of curvature are linear}]), i,e. they are in
$\mathcal{H}om_\mathcal{A}(\mathcal{E},\mathcal{E}\otimes_\mathcal{A}\Omega_D^2(\mathcal{A}))$, we get the following (using Lemma [\ref{induced inner-product}]),
\begin{eqnarray*}
&  & \sqrt{\mathcal{YM}}(\nabla)\\
& = & \sqrt{\langle\langle\varTheta\,,\,\varTheta\rangle\rangle}\\
& = & ||\varTheta_1\otimes 1+1\otimes\varTheta_2||\\
& \leq & ||\varTheta_1\otimes 1||+||1\otimes\varTheta_2||\\
& = & \sqrt{\sum_{i=1}^m\sum_{j=1}^n\,\langle\varTheta_1(\sigma_i)\otimes\mu_j\,,\,\varTheta_1(\sigma_i)\otimes\mu_j\rangle}+\sqrt{\sum_{i=1}^m
\sum_{j=1}^n\,\langle\sigma_i\otimes\varTheta_2(\mu_j)\,,\,\sigma_i\otimes\varTheta_2(\mu_j)\rangle}\\
& = & \sqrt{\sum_{i=1}^m\,cn\langle\varTheta_1(\sigma_i)\,,\,\varTheta_1(\sigma_i)\rangle Tr_\omega\left(|D_2|^{-\ell}\right)}+\sqrt{\sum_{j=1}^n\,cm
Tr_\omega\left(|D_1|^{-k}\right)\langle\varTheta_2(\mu_j)\,,\,\varTheta_2(\mu_j)\rangle}\\
& = & \sqrt{cnTr_\omega\left(|D_2|^{-\ell}\right)}\sqrt{\mathcal{YM}}(\nabla_1)+\sqrt{cmTr_\omega\left(|D_1|^{-k}\right)}\sqrt{\mathcal{YM}}(\nabla_2)\\
& = & \sqrt{\alpha}\sqrt{\mathcal{YM}}(\nabla_1)+\sqrt{\beta}\sqrt{\mathcal{YM}}(\nabla_2)
\end{eqnarray*}
where,
\begin{align*}
\alpha=cnTr_\omega\left(|D_2|^{-\ell}\right) &= \frac{n\Gamma(\frac{k}{2}+1)\Gamma(\frac{\ell}{2}+1)}{\Gamma\left(\frac{k+\ell}{2}+1\right)}
Tr_\omega\left(|D_2|^{-\ell}\right)\\
\beta=cmTr_\omega\left(|D_1|^{-k}\right) &= \frac{m\Gamma(\frac{k}{2}+1)\Gamma(\frac{\ell}{2}+1)}{\Gamma\left(\frac{k+\ell}{2}+1\right)}
Tr_\omega\left(|D_1|^{-k}\right)
\end{align*}
are two positive real constants. This concludes the proof.
\end{proof}

\begin{remark}\rm
We will reserve the notations for the constants $\alpha$ and $\beta$ throughout the rest of this article.
\end{remark}

\begin{corollary}
The additivity is stronger condition than the subadditivity.
\end{corollary}
\begin{proof}
The expression for $\varTheta$ in Propn. (\ref{the curvature}) and additivity of the Yang-Mills implies that
\begin{center}
$\mathcal{YM}(\nabla)=||\varTheta_1\otimes 1+1\otimes\varTheta_2||^2=\alpha||\varTheta_1||_{\mathcal{E}_1}^2+\beta||\varTheta_2||_{\mathcal{E}_2}^2\,.$
\end{center}
Therefore,
\begin{eqnarray*}
\sqrt{\mathcal{YM}}(\nabla) & = & \sqrt{\alpha||\varTheta_1||_{\mathcal{E}_1}^2+\beta||\varTheta_2||_{\mathcal{E}_2}^2}\\
& \leq & \sqrt{\alpha}||\varTheta_1||_{\mathcal{E}_1}+\sqrt{\beta}||\varTheta_2||_{\mathcal{E}_2}\\
& = & \sqrt{\alpha}\sqrt{\mathcal{YM}}(\nabla_1)+\sqrt{\beta}\sqrt{\mathcal{YM}}(\nabla_2)
\end{eqnarray*}
and hence, the additivity implies the subadditivity i,e. it is a stronger condition.
\end{proof}

\begin{proposition}\label{condition for additivity}
A necessary and sufficient condition for additivity of the Yang-Mills action functional is the following
\begin{center}
$Re\left(\left(\sum_{i=1}^m\overline{Tr_\omega\left((\varTheta_1(\sigma_i))_i|D_1|^{-k}\right)}\right)\left(\sum_{j=1}^nTr_\omega
\left((\varTheta_2(\mu_j))_j|D_2|^{-\ell}\right)\right)\right)=0$
\end{center}
where, $Tr_\omega$ denotes the Dixmier trace. Here, $\,Tr_\omega\left((\varTheta_1(\sigma_i))_i|D_1|^{-k}\right)$ means $\,Tr_\omega
\left(P_1v_i|D_1|^{-k}\right)$, where $P_1$ is the orthogonal projection onto the orthogonal complement of $\,\pi_1(d_1J_0^1(\mathcal{A}_1))$, and
$(\varTheta_1(\sigma_i))_i=[v_i]\in\Omega_{D_1}^2(\mathcal{A}_1)$ for $v_i\in\pi_1(\Omega^2(\mathcal{A}_1))$. Similar meaning for $\,Tr_\omega
\left((\varTheta_2(\mu_j))_i|D_2|^{-l}\right)$.
\end{proposition}
\begin{proof}
Let $\mathcal{E}_1=p_1\mathcal{A}_1^m$ and $\mathcal{E}_2=p_2\mathcal{A}_2^n$ be Hermitian f.g.p modules over $\mathcal{A}_1$ and $\mathcal{A}_2$
respectively, where the Hermitian structures are the induced canonical structure, and assume that $(\mathcal{A}_1,\mathcal{H}_1,D_1)$ is a $k$-summable
spectral triple and $(\mathcal{A}_2,\mathcal{H}_2,D_2)$ is a $\ell$-summable spectral triple. From the proof of Thm. (\ref{subadditivity}), using
Lemma (\ref{induced inner-product}), we get
\begin{eqnarray*}
\mathcal{YM}(\nabla) & = & \alpha \mathcal{YM}(\nabla_1)+\beta \mathcal{YM}(\nabla_2)\\
&  & +\sum_{i=1}^m\sum_{j=1}^n\,\langle\varTheta_1(\sigma_i)\otimes\mu_j\,,\,\sigma_i\otimes\varTheta_2(\mu_j)\rangle
+\langle\sigma_i\otimes\varTheta_2(\mu_j)\,,\,\varTheta_1(\sigma_i)\otimes\mu_j\rangle\\
& = & \alpha \mathcal{YM}(\nabla_1)+\beta \mathcal{YM}(\nabla_2)\\
&  & +\sum_{i=1}^m\sum_{j=1}^n\,\langle\varTheta_1(\sigma_i),\sigma_i\rangle\langle\mu_j,\varTheta_2(\mu_j)\rangle+
\langle\sigma_i,\varTheta_1(\sigma_i)\rangle\langle\varTheta_2(\mu_j),\mu_j\rangle\\
& = & \alpha \mathcal{YM}(\nabla_1)+\beta \mathcal{YM}(\nabla_2)\\
&  & +\sum_{i=1}^m\sum_{j=1}^n\,2Re\left(\overline{\langle\sigma_i,\varTheta_1(\sigma_i)\rangle}\langle\mu_j,\varTheta_2(\mu_j)\rangle\right)\\
& = & \alpha \mathcal{YM}(\nabla_1)+\beta \mathcal{YM}(\nabla_2)\\
&  & +\sum_{i=1}^m\sum_{j=1}^n\,2Re\left(\overline{Tr_\omega((\varTheta_1(\sigma_i))_i|D_1|^{-k})}\,Tr_\omega\left((\varTheta_2(\mu_j))_j|D_2|^{-\ell}\right)\right)\\
& = & \alpha \mathcal{YM}(\nabla_1)+\beta \mathcal{YM}(\nabla_2)\\
&  & +\,2Re\left(\left(\sum_{i=1}^m\overline{Tr_\omega\left((\varTheta_1(\sigma_i))_i|D_1|^{-k}\right)}\right)\left(\sum_{j=1}^n
Tr_\omega\left((\varTheta_2(\mu_j))_j|D_2|^{-\ell}\right)\right)\right)
\end{eqnarray*}
Here, $(\varTheta_1(\sigma_i))_i$ is the $i$-th co-ordinate of $\varTheta_1(\sigma_i)\in\left(\Omega_{D_1}^2(\mathcal{A}_1)\right)^m$ and
$(\varTheta_2(\mu_j))_j$ is the $j$-th co-ordinate of $\varTheta_2(\mu_j)\in\left(\Omega_{D_2}^2(\mathcal{A}_2)\right)^n$. This is because
for $j=1,2$, range of $\varTheta_j$ lies in $\,p_j\mathcal{A}_j^{k_j}\otimes_{\mathcal{A}_j}\Omega_{D_j}^2(\mathcal{A}_j)$ which is contained in
$\left(\Omega_{D_j}^2(\mathcal{A}_j)\right)^{k_j}$, with $k_j=m$ if $j=1$ and $k_j=n$ if $j=2$. The meaning of the complex number $Tr_\omega
\left((\varTheta_1(\sigma_i))_i|D_1|^{-k}\right)$ is then clear. Choose any representative $v_i\in\pi_1\left(\Omega^2(\mathcal{A}_1)\right)$
of $(\varTheta_1(\sigma_i))_i\in\Omega_{D_1}^2(\mathcal{A}_1)$. Then, $Tr_\omega\left((\varTheta_1(\sigma_i))_i|D_1|^{-k}\right)$ means the complex
number $Tr_\omega\left(P_1v_i|D_1|^{-k}\right)$, where $P_1$ is the orthogonal projection onto the orthogonal complement of $\pi_1(d_1J_0^1
(\mathcal{A}_1))$. Similar meaning for $Tr_\omega\left((\varTheta_2(\mu_j))_j|D_2|^{-l}\right)$. Let $\,\xi,\eta\,$ denote the following complex
numbers
\begin{align*}
\xi &= \sum_{i=1}^mTr_\omega\left((\varTheta_1(\sigma_i))_i|D_1|^{-k}\right)\,,\\
\eta &= \sum_{j=1}^nTr_\omega\left((\varTheta_2(\mu_j))_j|D_2|^{-\ell}\right)\,.
\end{align*}
Hence, $\,Re\left(\overline\xi\eta\right)=0$ is a necessary and sufficient condition for additivity of the Yang-Mills functional, and this
condition depends only on the individual spectral triples.
\end{proof}

An instance of additivity of the Yang-Mills functional is shown in the next section. If the Yang-Mills functional becomes additive then it is
natural to ask when critical points (Def. [\ref{critical point}]) on the individual spectral triples give rise to a critical point on the product
spectral triple. We obtain a necessary and sufficient condition for this.

\begin{proposition}\label{necessary condition}
A necessary condition for $\nabla$ to be a critical point for the Yang-Mills functional under additivity is that both $\nabla_1,\nabla_2$ must be
critical points for the Yang-Mills functional on the individual spectral triple.
\end{proposition}
\begin{proof}
Choose $\mu_j\in\mathcal{H}om_{\mathcal{A}_j}(\mathcal{E}_j,\mathcal{E}_j\otimes_{\mathcal{A}_j}\Omega_{D_j}^1(\mathcal{A}_j))$ for $j=1,2$. Define
$\mu=\mu_1\otimes 1+1\otimes\mu_2$. Then $\,\mu\in\mathcal{H}om_\mathcal{A}(\mathcal{E},\mathcal{E}\otimes_\mathcal{A}\Omega_D^1(\mathcal{A}))$. If
$\nabla$ is a critical point for the Yang-Mills functional then, we have
\begin{eqnarray*}
0 & = & \frac{d}{dt}|_{t=0}\,\mathcal{YM}(\nabla+t\mu)\\
& = & \frac{d}{dt}|_{t=0}\,\mathcal{YM}(\nabla_1\otimes 1+1\otimes\nabla_2+t(\mu_1\otimes 1+1\otimes\mu_2))\\
& = & \frac{d}{dt}|_{t=0}\,\mathcal{YM}((\nabla_1\otimes 1+t\mu_1\otimes 1)+(1\otimes\nabla_2+t\otimes\mu_2))\\
& = & \frac{d}{dt}|_{t=0}\,\mathcal{YM}((\nabla_1+t\mu_1)\otimes 1+1\otimes(\nabla_2+t\mu_2))\\
& = & \frac{d}{dt}|_{t=0}\,\alpha \mathcal{YM}(\nabla_1+t\mu_1)+\frac{d}{dt}|_{t=0}\,\beta \mathcal{YM}(\nabla_2+t\mu_2)\\
& = & \alpha\frac{d}{dt}|_{t=0}\,\mathcal{YM}(\nabla_1+t\mu_1)+\beta\frac{d}{dt}|_{t=0}\,\mathcal{YM}(\nabla_2+t\mu_2)
\end{eqnarray*}
(by Th. [\ref{subadditivity}]). Since, the range of $\mathcal{YM}$ is $\mathbb{R}_{\geq 0}$ and $\alpha,\beta$ are positive real constants, we see
that $\nabla_1,\nabla_2$ both are critical points for the Yang-Mills functional.
\end{proof}

Recall from Def. (\ref{critical point}) that a connection $\nabla$ on a Hermitian f.g.p module $\mathcal{E}$ is a critical point for the Yang-Mills
functional if and only if $\,\langle\langle[\nabla,\mu]\,,\,\varTheta\rangle\rangle=0,\,\,\forall\,\mu\in\mathcal{H}om_\mathcal{A}(\mathcal{E},\mathcal{E}
\otimes_\mathcal{A}\Omega_D^1(\mathcal{A}))$. Here, $[\nabla,\mu]=\widetilde{\nabla}\circ\mu+(1\otimes\Pi)\circ(\mu\otimes 1)\circ\nabla$. In our
situation, $\mathcal{A}=\mathcal{A}_1\otimes\mathcal{A}_2$ and $\mathcal{E}=\mathcal{E}_1\otimes\mathcal{E}_2$. Using Lemma (\ref{first form}),
any $\,\mu\in\mathcal{H}om_\mathcal{A}(\mathcal{E},\mathcal{E}\otimes_\mathcal{A}\Omega_D^1(\mathcal{A}))$ can be written as $\,\mu=\mu_1\oplus\mu_2$ where,
\begin{align*}
\mu_1=Pr_1\circ\mu:\mathcal{E}_1\otimes\mathcal{E}_2 &\longrightarrow \left(\mathcal{E}_1\otimes_{\mathcal{A}_1}\Omega_{D_1}^1(\mathcal{A}_1)\right)
\otimes\mathcal{E}_2\,,\\
\mu_2=Pr_2\circ\mu:\mathcal{E}_1\otimes\mathcal{E}_2 &\longrightarrow \mathcal{E}_1\otimes\left(\mathcal{E}_2\otimes_{\mathcal{A}_2}\Omega_{D_2}^1
(\mathcal{A}_2)\right)\,.
\end{align*}
For $\nabla_j\in C(\mathcal{E}_j),\,j=1,2,$ we define
\begin{align*}
[\nabla_1\otimes 1,\mu_1] &:= (\widetilde{\nabla_1}\otimes 1)\circ\mu_1+(1\otimes\Pi)\circ(\mu_1\otimes 1)\circ(\nabla_1\otimes 1)\,,\\
[1\otimes\nabla_2,\mu_2] &:= (1\otimes\widetilde{\nabla_2})\circ\mu_2+(1\otimes\Pi)\circ(\mu_2\otimes 1)\circ(1\otimes\nabla_2)\,.
\end{align*}

\begin{lemma}
We have
\begin{align*}
[\nabla_1\otimes 1,\mu_1] &\in \mathcal{H}om_\mathcal{A}\left(\mathcal{E}_1\otimes\mathcal{E}_2\,,\,(\mathcal{E}_1
\otimes_{\mathcal{A}_1}\Omega_{D_1}^2(\mathcal{A}_1))\otimes_{\mathcal{A}_2}\mathcal{E}_2\right)\\
[1\otimes\nabla_2,\mu_2] &\in \mathcal{H}om_\mathcal{A}\left(\mathcal{E}_1\otimes\mathcal{E}_2\,,\,\mathcal{E}_1
\otimes_{\mathcal{A}_1}(\mathcal{E}_2\otimes_{\mathcal{A}_2}\Omega_{D_2}^2(\mathcal{A}_2))\right)
\end{align*}
i,e. both are elements of $\,\mathcal{H}om_\mathcal{A}(\mathcal{E},\mathcal{E}\otimes_\mathcal{A}\Omega_D^2(\mathcal{A}))$.
\end{lemma}
\begin{proof}
For $\xi_1\in\mathcal{E}_1$ and $\xi_2\in\mathcal{E}_2$, let
\begin{align*}
\mu_1(\xi_1\otimes\xi_2) &= \sum_j\xi_{1j}\otimes\omega_{1j}\otimes\xi_{2j}\in\left(\mathcal{E}_1\otimes_{\mathcal{A}_1}
\Omega_{D_1}^1\right)\otimes\mathcal{E}_2\\
\nabla_1(\xi_1) &= \sum_j\widetilde{\xi_{1j}}\otimes\widetilde{\omega_{1j}}\in\mathcal{E}_1\otimes_{\mathcal{A}_1}\Omega_{D_1}^1\,.
\end{align*}
Then,
\begin{eqnarray*}
&  & (\widetilde{\nabla_1}\otimes 1)\circ\mu_1(\xi_1a\otimes\xi_2b)\\
& = & \sum_j(\widetilde{\nabla_1}\otimes 1)\left((\xi_{1j}\otimes\omega_{1j}\otimes\xi_{2j})(a\otimes b)\right)\\
& = & \sum_j(\widetilde{\nabla_1}\otimes 1)(\xi_{1j}\otimes\omega_{1j}a\otimes\xi_{2j}b)\\
& = & \sum_j(\nabla_1(\xi_{1j})\omega_{1j}a+\xi_{1j}\otimes d_1(\omega_{1j}a))\otimes\xi_{2j}b\\
& = & \sum_j(\nabla_1(\xi_{1j})\omega_{1j}\otimes\xi_{2j})(a\otimes b)+((\xi_{1j}\otimes\xi_{2j})\otimes(d_1(\omega_{1j})\otimes 1)) (a\otimes b)\\
&   & \quad\,\,-((\xi_{1j}\otimes\xi_{2j})\otimes(\omega_{1j}d_1a\otimes 1))(1\otimes b)
\end{eqnarray*}
and
\begin{eqnarray*}
&  & (1\otimes\Pi)\circ(\mu_1\otimes 1)\circ(\nabla_1\otimes 1)(\xi_1a\otimes\xi_2b)\\
& = & (1\otimes\Pi)\circ(\mu_1\otimes 1)((\nabla_1(\xi_1)a+\xi_1\otimes d_1a)\otimes\xi_2b)\\
& = & (1\otimes\Pi)\circ(\mu_1\otimes 1)\left(\left(\sum_j(\widetilde{\xi_{1j}}\otimes\xi_2)\otimes(\widetilde{\omega_{1j}}\otimes 1)\right)(a\otimes b)+
((\xi_1\otimes\xi_2)\otimes(d_1a\otimes 1))(1\otimes b)\right)\\
& = &  \sum_j(1\otimes\Pi)\circ(\mu_1\otimes 1)\left(\left((\widetilde{\xi_{1j}}\otimes\xi_2)\otimes(\widetilde{\omega_{1j}}\otimes 1)\right)\right)(a\otimes b)\\
&  & \quad+(1\otimes\Pi)((\xi_{1j}\otimes\omega_{1j}\otimes\xi_{2j})\otimes d_1a)(1\otimes b)\\
& = & \sum_j(1\otimes\Pi)\circ(\mu_1\otimes 1)\left((\widetilde{\xi_{1j}}\otimes\xi_2)\otimes(\widetilde{\omega_{1j}}\otimes 1)\right)(a\otimes b)
+((\xi_{1j}\otimes\xi_{2j})\otimes(\omega_{1j}d_1a\otimes 1))(1\otimes b)
\end{eqnarray*}
Adding these two we see that $[\nabla_1\otimes 1,\mu_1]$ is $\mathcal{A}$-linear. Similarly, one can show for $[1\otimes\nabla_2,\mu_2]$.
\end{proof}

\begin{proposition}\label{sufficient condition}
Given two spectral triples $(\mathcal{A}_j,\mathcal{H}_j,D_j,\gamma_j),\,j=1,2$, and Hermitian f.g.p modules $\mathcal{E}_j$ over $\mathcal{A}_j$,
if $\,\nabla_j\in C(\mathcal{E}_j)$ satisfy the following equation
\begin{center}
$\langle\langle\,[\nabla_1\otimes 1,\mu_1]\,,\,\varTheta_1\otimes 1\rangle\rangle+\langle\langle\,[1\otimes\nabla_2,\mu_2]\,,\,1\otimes\varTheta_2
\rangle\rangle=0$
\end{center}
for all $\mu\in\mathcal{H}om_\mathcal{A}(\mathcal{E},\mathcal{E}\otimes_\mathcal{A}\Omega_D^1(\mathcal{A}))$, then $\nabla$ is a critical point for
the Yang-Mills functional on product spectral triple. Moreover, the converse is also true.
\end{proposition}
\begin{proof}
We have $\nabla=\nabla_1\otimes 1+1\otimes\nabla_2$ and $\varTheta=\varTheta_1\otimes 1+1\otimes\varTheta_2$ by Propon.
(\ref{the connection}\,,\ref{the curvature}). Now, for the standard basis element $\sigma_i\otimes\tau_j,\,i=1,\ldots,m,\,j=1,\ldots,n$,
of the free $\mathcal{A}$-module $(\mathcal{A}_1\otimes\mathcal{A}_2)^{mn}=\mathcal{A}_1^m\otimes\mathcal{A}_2^n\,$, suppose that
\begin{align*}
\mu(\sigma_i\otimes\tau_j) &= \sum_k\xi_{ijk}\otimes\omega_{ijk}\,\,\in\mathcal{E}\otimes_\mathcal{A}\Omega_D^1(\mathcal{A})\,,\\
\nabla(\sigma_i\otimes\tau_j) &= \sum_k\eta_{ijk}\otimes v_{ijk}\,\,\in\mathcal{E}\otimes_\mathcal{A}\Omega_D^1(\mathcal{A})\,.
\end{align*}
Then,
\begin{eqnarray*}
&  & [\nabla,\mu](\sigma_i\otimes\tau_j)\\
& = & \sum_k\nabla(\xi_{ijk})\omega_{ijk}+\xi_{ijk}\otimes d\omega_{ijk}+\mu(\eta_{ijk})v_{ijk}\\
& = & \sum_k(\nabla_1(\xi_{ijk1})\otimes\xi_{ijk2}+\xi_{ijk1}\otimes\nabla_2(\xi_{ijk2}))\omega_{ijk}+\xi_{ijk}\otimes d\omega_{ijk}+\mu(\eta_{ijk})v_{ijk}\\
& = & \sum_k\nabla_1(\xi_{ijk1})\omega_{ijk1}\otimes\xi_{ijk2}a_2-\nabla_1(\xi_{ijk1})(a_1\otimes\omega_{ijk2})\otimes\xi_{ijk2}+\xi_{ijk1}\otimes
\nabla_2(\xi_{ijk2})(\omega_{ijk1}\otimes a_2)\\
&  & \quad\,\,+\xi_{ijk1}a_1\otimes\nabla_2(\xi_{ijk2})\omega_{ijk2}+\xi_{ijk1}\otimes\xi_{ijk2}a_2\otimes d_1\omega_{ijk1}+\xi_{ijk1}a_1\otimes
\xi_{ijk2}\otimes d_2\omega_{ijk2}\\
&  & \quad\,\,+\xi_{ijk1}\otimes\xi_{ijk2}\otimes(\omega_{ijk1}\otimes d_2a_2-d_1a_1\otimes\omega_{ijk2})+\mu(\eta_{ijk})v_{ijk}\\
\end{eqnarray*}
as an element of $\,\mathcal{E}\otimes_\mathcal{A}\Omega_D^2(\mathcal{A})$, by using Lemma (\ref{2nd differential}\,,\,\ref{product map}).
So, for $\varTheta=\varTheta_1\otimes 1+1\otimes\varTheta_2\,$, we have the following,
\begin{eqnarray*}
&  & \langle\langle\,[\nabla,\mu]\,,\,\varTheta\,\rangle\rangle\\
& = & \sum_{i=1}^m\sum_{j=1}^n\,\langle[\nabla,\mu](\sigma_i\otimes\tau_j)\,,\,\varTheta(\sigma_i\otimes\tau_j)\rangle\\
& = & \sum_{i=1}^m\sum_{j=1}^n\,\langle[\nabla,\mu](\sigma_i\otimes\tau_j)\,,\,\varTheta_1(\sigma_i)\otimes\tau_j\rangle+\langle[\nabla,\mu]
(\sigma_i\otimes\tau_j)\,,\,\sigma_i\otimes\varTheta_2(\tau_j)\rangle\\
& = & \sum_{i=1}^m\sum_{j=1}^n\sum_k\,\langle\nabla_1(\xi_{ijk1})\omega_{ijk1}\otimes\xi_{ijk2}a_2+\xi_{ijk1}\otimes\xi_{ijk2}a_2\otimes
d_1\omega_{ijk1}+\mu(\eta_{ijk})v_{ijk}\,,\varTheta_1(\sigma_i)\otimes\tau_j\rangle\\
&   & \quad\quad\quad+\langle\xi_{ijk1}a_1\otimes\nabla_2(\xi_{ijk2})\omega_{ijk2}+\xi_{ijk1}a_1\otimes\xi_{ijk2}\otimes d_2\omega_{ijk2}+
\mu(\eta_{ijk})v_{ijk}\,,\,\sigma_i\otimes\varTheta_2(\tau_j)\rangle\\
& = & \sum_{i=1}^m\sum_{j=1}^n\sum_k\,\langle(\nabla_1(\xi_{ijk1})\omega_{ijk1}+\xi_{ijk1}\otimes d_1\omega_{ijk1})\otimes\xi_{ijk2}a_2\,,\,
\varTheta_1(\sigma_i)\otimes\tau_j\rangle\\
&  & \,\,\,\,\,+\langle\mu(\eta_{ijk})v_{ijk}\,,\,\varTheta(\sigma_i\otimes\tau_j)\rangle+\langle\,\xi_{ijk1}a_1\otimes(\nabla_2(\xi_{ijk2})
\omega_{ijk2}+\xi_{ijk2}\otimes d_2\omega_{ijk2})\,,\,\sigma_i\otimes\varTheta_2(\tau_j)\rangle\\
& = & \sum_{i=1}^m\sum_{j=1}^n\,\langle((\widetilde{\nabla_1}\otimes 1)\circ Pr_1\circ\mu)(\sigma_i\otimes\tau_j)\,,\,\varTheta_1(\sigma_i)
\otimes\tau_j\rangle+\langle\mu(\eta_{ijk})v_{ijk}\,,\,\varTheta(\sigma_i\otimes\tau_j)\rangle\\
&   & \quad\quad\quad+\langle((1\otimes\widetilde{\nabla_2})\circ Pr_2\circ\mu)(\sigma_i\otimes\tau_j)\,,\,\sigma_i\otimes\varTheta_2(\tau_j)\rangle\\
& = & \langle\langle(\widetilde{\nabla_1}\otimes 1)\circ\mu_1\,,\,\varTheta_1\otimes 1\rangle\rangle+\langle\langle(1\otimes\widetilde{\nabla_2})
\circ\mu_2\,,\,1\otimes\varTheta_2\rangle\rangle+\langle\langle(1\otimes\Pi)\circ(\mu\otimes 1)\circ\nabla\,,\,\varTheta\rangle\rangle
\end{eqnarray*}
where, $\mu_1=Pr_1\circ\mu$ and $\mu_2=Pr_2\circ\mu$. Now,
\begin{eqnarray*}
&  & \langle\langle(1\otimes\Pi)\circ(\mu\otimes 1)\circ\nabla\,,\,\varTheta\rangle\rangle\\
& = & \sum_{i,j}\langle(1\otimes\Pi)\circ(\mu\otimes 1)\circ\nabla(\sigma_i\otimes\tau_j)\,,\,\varTheta(\sigma_i\otimes\tau_j)\rangle\\
& = & \sum_{i,j}\langle(1\otimes\Pi)\circ(\mu\otimes 1)\circ(\nabla_1(\sigma_i)\otimes\tau_j+\sigma_i\otimes\nabla_2(\tau_j))\,,\,\varTheta_1
(\sigma_i)\otimes\tau_j+\sigma_i\otimes\varTheta_2(\tau_j)\rangle\\
& = & \sum_{i,j}\langle(1\otimes\Pi)\circ\mu_1(\nabla_1(\sigma_i)\otimes\tau_j)\,,\,\varTheta_1(\sigma_i)\otimes\tau_j\rangle
+\langle(1\otimes\Pi)\circ\mu_2(\sigma_i\otimes\nabla_2(\tau_j)),\,\sigma_i\otimes\varTheta_2(\tau_j)\rangle\\
& = & \langle\langle(1\otimes\Pi)\circ\mu_1\circ(\nabla_1\otimes 1)\,,\,\varTheta_1\otimes 1\rangle\rangle+\langle\langle(1\otimes\Pi)\circ\mu_2
\circ(1\otimes\nabla_2)\,,\,1\otimes\varTheta_2\rangle\rangle
\end{eqnarray*}
Hence, we have
\begin{eqnarray*}
\langle\langle\,[\nabla,\mu]\,,\,\varTheta\,\rangle\rangle=\langle\langle\,[\nabla_1\otimes 1,\mu_1]\,,\,\varTheta_1\otimes 1\rangle\rangle+
\langle\langle\,[1\otimes\nabla_2,\mu_2]\,,\,1\otimes\varTheta_2\rangle\rangle
\end{eqnarray*}
and this concludes the proof.
\end{proof}

Combining Propn. (\ref{necessary condition}) and (\ref{sufficient condition}) we conclude the following final theorem.

\begin{theorem}
If the Yang-Mills functional is additive then a necessary and sufficient condition for $\nabla$ to be a critical point for the Yang-Mills functional
on the product spectral triple is that both $\nabla_1,\nabla_2$ are critical points for the Yang-Mills functional on the individual spectral triple,
and they satisfy the following equation
\begin{center}
$\langle\langle[\nabla_1\otimes 1,\mu_1]\,,\,\varTheta_1\otimes 1\rangle\rangle+\langle\langle[1\otimes\nabla_2,\mu_2]\,,\,1\otimes\varTheta_2\rangle\rangle=0$
\end{center}
for all $\,\mu\in\mathcal{H}om_\mathcal{A}(\mathcal{E},\mathcal{E}\otimes\Omega_D^1(\mathcal{A}))$, where $\,\mu_1=Pr_1\circ\mu$ and $\mu_2=Pr_2\circ\mu\,$.
\end{theorem}
\bigskip


\section{An instance of additivity~: The case of noncommutative tori}

In this section we provide an instance of additivity of the Yang-Mills functional for the case of noncommutative tori.
\begin{definition}
Let $\Theta\in M_n(\mathbb{R})$ be any $n\times n$ real skew-symmetric matrix. Denote by $\mathcal{A}_\Theta$ the universal $C^*$-algebra
generated by $\,n$ unitaries $\,U_1,\ldots,U_n\,$ satisfying
\begin{center}
$U_kU_m=e^{2\pi i\Theta_{mk}}U_mU_k$
\end{center}
for $k,m\in\{1,\ldots,n\}$.
\end{definition}

{\bf Action of the Lie group $\mathbb{T}^n$:} On $\mathcal{A}_\Theta$, the compact connected Lie group $\,\mathbb{T}^n$ acts as follows:
\begin{eqnarray*}
\alpha_{(z_1,\ldots,z_n)}(U_k)=z_kU_k\,,\,k=1,\ldots,n.
\end{eqnarray*}

{\bf The smooth subalgebra $\mathcal{A}_\Theta^\infty$:} The smooth subalgebra of $\mathcal{A}_\Theta$ under this action is given by
\begin{eqnarray*}
\mathcal{A}_\Theta^\infty:=\left\{\sum\,a_{\textbf{r}}\,U^{\textbf{r}}:\{a_{\textbf{r}}\}\in\mathbb{S}(\mathbb{Z}^{n})\,,\,\textbf{r}=
(r_1,\ldots,r_n)\in\mathbb{Z}^n\right\}\,,
\end{eqnarray*}
where, $\,\mathbb{S}(\mathbb{Z}^{n})$ denotes vector space of multisequences $(a_\textbf{r})$ that decay faster than inverse of any polynomial 
in $\,\textbf{r}=(r_1,\ldots,r_n)$. This is a unital subalgebra of $\mathcal{A}_\Theta$ stable under the holomorphic functional calculus (\cite{GVF}),
and called the noncommutative $n$-torus.

{\bf The Trace:} The subalgebra $\mathcal{A}_\Theta^\infty$ is equipped with a unique $\,\mathbb{T}^n$-invariant tracial state given by
$\tau(a)=a_\textbf{0}\,,$ where $\textbf{0}=(0,\ldots,0)\in\mathbb{Z}^n$.

{\bf The G.N.S. Hilbert space:} The Hilbert space $L^2(\mathcal{A}_\Theta^\infty,\tau)$ obtained by applying the G.N.S. construction to
$\tau$ can be identified with $\,\ell^2(\mathbb{Z}^n)$ (\cite{Rfl}).

{\bf The spectral triple:} Consider the irreducible representation of $\mathbb{C}\ell(n)$ on $\mathbb{C}^N$, where $\,N=2^{\lfloor n/2\rfloor}$.
Then, there are $n$ many Clifford gamma matrices $\gamma_1,\ldots ,\gamma_n$ in $M_N(\mathbb{C})$ satisfying $\gamma_r\gamma_s+\gamma_s\gamma_r=
2\delta_{rs}\,,\,r,s\in\{1,\ldots,n\}\,$, where $\delta_{rs}$ denotes the Kronecker delta function. Consider the densely defined unbounded
symmetric operator $D_\Theta:=\sum_{j=1}^n\delta_j\otimes\gamma_j$ where,
\begin{center}
$\delta_j\left(\sum_{\textbf{r}} a_{\textbf{r}}U^{\textbf{r}}\right):=\sum_{\textbf{r}}2\pi ir_ja_{\textbf{r}}U^{\textbf{r}}\,.$
\end{center}
It is known that $D_\Theta\,$ is self-adjoint with compact resolvent, acting on $\mathcal{H}_\Theta=L^2(\mathcal{A}_\varTheta^\infty,\tau)\otimes
\mathbb{C}^N,\,N=2^{\lfloor n/2\rfloor}$. Moreover, $|D_\Theta|^{-n}$ lies in the Dixmier ideal $\mathcal{L}^{(1,\infty)}$ with $Tr_\omega
(|D_\Theta|^{-n})=2N\pi^{n/2}/(n(2\pi)^n\Gamma(n/2))\,$ (see \cite{GVF}, Page $545$). The tuple $(\mathcal{A}_\Theta^\infty,\mathcal{H}_\Theta,D_\Theta)$
gives us a $n$-summable spectral triple on $\mathcal{A}_\Theta^\infty$. If $n$ is even then this is an even spectral triple and the grading operator
comes from the irreducible representation of $\mathbb{C}\ell(n)$ on $\mathbb{C}^N$.
\medskip

We will be working with $\mathcal{A}_\Theta^\infty$ and denote it simply by $\mathcal{A}_\Theta$ for notational brevity. Consider the product
$\mathcal{A}_\Theta\otimes\mathcal{A}_\Phi$, where $\mathcal{A}_\Theta$ is a noncommutative $n$-torus and $\mathcal{A}_\Phi$ is a noncommutative
$m$-torus. It is known that (Proposition $5.1$ and $5.3$ in \cite{CG1}),
\begin{align*}
\Omega_{D_\Theta}^1(\mathcal{A}_\Theta)=\mathcal{A}_\Theta^n\quad &, \quad\Omega_{D_\Phi}^1(\mathcal{A}_\Phi)=\mathcal{A}_\Phi^m\,,\\
\pi_1\left(\Omega^2(\mathcal{A}_\Theta)\right)=\mathcal{A}_\Theta^{1+n(n-1)/2}\quad &, \quad\pi_2\left(\Omega^2(\mathcal{A}_\Phi)\right)=\mathcal{A}_\Phi^{1+m(m-1)/2}\,,\\
\pi_1\left(d_1J_0^1(\mathcal{A}_\Theta)\right)=\mathcal{A}_\Theta\quad &, \quad\pi_2\left(d_2J_0^1(\mathcal{A}_\Phi)\right)=\mathcal{A}_\Phi\,,\\
\Omega_{D_\Theta}^2(\mathcal{A}_\Theta)=\mathcal{A}_\Theta^{n(n-1)/2}\quad &, \quad\Omega_{D_\Phi}^2(\mathcal{A}_\Phi)=\mathcal{A}_\Phi^{m(m-1)/2}\,.
\end{align*}
If $(\mathcal{A}_\Theta,\mathcal{H}_\Theta,D_\Theta)$ is not an even spectral triple (unless $n$ is even) we apply the process described in
point (\ref{from odd to even}) in Section $(2)$ to make it even with grading operator $\gamma$. Let $\,D=D_\Theta\otimes 1+\gamma\otimes D_\Phi$.
Intuitively, one can guess that the Yang-Mills functional is going to be additive in this case. The reason is that $\mathcal{A}_\Theta\otimes
\mathcal{A}_\Phi$ can be identified with a noncommutative $(n+m)$-torus $\mathcal{A}_\Psi$ for an obvious choice of $\Psi$, and $D$ becomes
$D_\Psi$ acting on $\mathcal{H}_\Psi=\ell^2(\mathbb{Z}^{n+m})\otimes\mathbb{C}^{2^{\lfloor(n+m)/2\rfloor}}$. Hence, both $\Omega_{D_\Psi}^1
(\mathcal{A}_\Psi)$ and $\,\Omega_{D_\Psi}^2(\mathcal{A}_\Psi)$ are free modules of rank $(n+m)$ and $(n+m)(n+m-1)/2$ respectively. So, the
Yang-Mills functional on a Hermitian f.g.p module $\mathcal{E}=p\mathcal{A}_\Psi^q$, with $p\in M_q(\mathcal{A}_\Psi)$ a projection, is given by
\begin{center}
$\mathcal{YM}(\nabla)=\sum_{1\leq i<j\leq n+m}\tau_q([\nabla_i,\nabla_j]^*[\nabla_i,\nabla_j])$
\end{center}
where, $\tau_q$ denotes the extended trace $\tau\otimes Trace$ on
$M_q(\mathcal{A}_\Psi)$ (see Proposition $5.12$ in \cite{CG1} for detail). This expression actually proves the additivity of the Yang-Mills
functional but we go through little detail to see why our hypothesis in Section $(3)$, or equivalently Lemma (\ref{2 form for product system}),
is justified in this case.

\begin{proposition}
$\Omega_D^2(\mathcal{A}_\Psi)\cong\,\Omega_{D_\Theta}^2(\mathcal{A}_\Theta)\otimes\mathcal{A}_\Phi\bigoplus\mathcal{A}_\Theta\otimes\Omega_{D_2}^2
(\mathcal{A}_\Phi)\bigoplus\Omega_{D_\Theta}^1(\mathcal{A}_\Theta)\otimes\Omega_{D_\Phi}^1(\mathcal{A}_\Phi)$ as $\mathcal{A}_\Psi=\mathcal{A}_\Theta
\otimes\mathcal{A}_\Phi$-bimodules.
\end{proposition}
\begin{proof}
One can conclude this by comparing the free module (over $\mathcal{A}_\Psi$) dimensions of both sides. Since, $\mathcal{A}_\Psi$ is a noncommutative
$(n+m)$-torus, $\Omega_D^2(\mathcal{A}_\Psi)$ has dimension $(n+m)(n+m-1)/2$ as free module over $\mathcal{A}_\Psi$ (Proposition $5.3$ in \cite{CG1}).
The dimension of $\Omega_{D_\Theta}^1(\mathcal{A}_\Theta)\otimes\Omega_{D_\Phi}^1(\mathcal{A}_\Phi)$ is $\,nm$ as free module over
$\mathcal{A}_\Psi$ (Proposition $5.1$ in \cite{CG1}). Therefore, by Lemma (\ref{numerator of 2nd form}) we see that
\begin{center}
$\frac{\pi(\Omega^2(\mathcal{A}_\Theta))\otimes\mathcal{A}_\Phi+\mathcal{A}_\Theta\otimes\pi(\Omega^2(\mathcal{A}_\Phi))}{\pi(dJ_0^1(\mathcal{A}_\Psi))}
\subseteq\frac{\pi(\Omega^2(\mathcal{A}_\Psi))}{\pi(dJ_0^1(\mathcal{A}_\Psi))}=\Omega_D^2(\mathcal{A}_\Psi)$
\end{center}
must be a free module with dimension $(n+m)(n+m-1)/2-nm=n(n-1)/2+m(m-1)/2$. Since, $\Omega_{D_\Theta}^2(\mathcal{A}_\Theta)\otimes
\mathcal{A}_\Phi\bigoplus\mathcal{A}_\Theta\otimes\Omega_{D_2}^2(\mathcal{A}_\Phi)$ is also a free module of dimension $\,n(n-1)/2+m(m-1)/2$,
we have a canonical isomorphism
\begin{center}
$\frac{\pi(\Omega^2(\mathcal{A}_\Theta))\otimes\mathcal{A}_\Phi+\mathcal{A}_\Theta\otimes\pi(\Omega^2(\mathcal{A}_\Phi))}{\pi(dJ_0^1(\mathcal{A}_\Psi))}
\,\cong\,\Omega_{D_\Theta}^2(\mathcal{A}_\Theta)\otimes\mathcal{A}_\Phi\bigoplus\mathcal{A}_\Theta\otimes\Omega_{D_2}^2(\mathcal{A}_\Phi)$
\end{center}
of $\mathcal{A}_\Psi$-bimodules, and this concludes the proof.
\end{proof}

\begin{remark}\rm
The proof of above Lemma explains Remark $(\ref{a counter example})$ in Section $(3)$. We see that $\pi(dJ_0^1(\mathcal{A}_\Psi))$ is a free module
over $\mathcal{A}_\Psi$ of rank $1$, whereas $\,\pi_1(d_1J_0^1(\mathcal{A}_\Theta))\otimes\mathcal{A}_\Phi\bigoplus\mathcal{A}_\Theta\otimes
\pi_2(d_2J_0^1(\mathcal{A}_\Phi))$ is a free module over $\mathcal{A}_\Psi$ of rank $2$.
\end{remark}

\begin{proposition}
The Yang-Mills functional is additive in this case.
\end{proposition}
\begin{proof}
Let $\,\mathcal{E}_1=p_1\mathcal{A}_\Theta^{q_n}$ and $\mathcal{E}_2=p_2\mathcal{A}_\Phi^{q_m}$ be two Hermitian f.g.p modules over
$\mathcal{A}_\Theta\,(n\,\,torus)$ and $\mathcal{A}_\Phi\,(m\,\,torus)$ respectively. Let $\nabla_1\in C(\mathcal{E}_1)$ and $\nabla_2\in
C(\mathcal{E}_2)$ be two compatible connections. Since, $\Omega_{D_\Theta}^1(\mathcal{A}_\Theta)$ and $\Omega_{D_\Phi}^1(\mathcal{A}_\Phi)$ both
are free modules of rank $n$ and $m$ respectively, we have
\begin{align*}
\nabla_1(\xi) &= \sum_{j=1}^n\nabla_{1j}(\xi)\otimes\sigma_j\\
\nabla_2(\eta) &= \sum_{k=1}^m\nabla_{2k}(\eta)\otimes\mu_k
\end{align*}
for $\mathbb{C}$-linear maps $\nabla_{1j}:\mathcal{E}_1\longrightarrow\mathcal{E}_1$ and $\nabla_{2k}:\mathcal{E}_2\longrightarrow\mathcal{E}_2$.
Here, $\{\sigma_1,\ldots,\sigma_n\}$ is the standard basis of the free module $\Omega_{D_\Theta}^1(\mathcal{A}_\Theta)$ over $\mathcal{A}_\Theta$
and $\{\mu_1,\ldots,\mu_m\}$ is that of $\Omega_{D_\Phi}^1(\mathcal{A}_\Phi)$ over $\mathcal{A}_\Phi$. It is known that the Yang-Mills functional
on $\mathcal{E}_1=p_1\mathcal{A}_\Theta^{q_n}$ is given by
\begin{eqnarray}\label{eqn 1}
\mathcal{YM}(\nabla_1) & = & \sum_{1\leq i<j\leq n}\tau_{q_n}([\nabla_{1i},\nabla_{1j}]^*[\nabla_{1i},\nabla_{1j}])
\end{eqnarray}
where, $\tau_{q_n}$ denotes the trace $\tau_\Theta\otimes Trace$ on $M_{q_n}(\mathcal{A}_\Theta)$ (see Proposition $5.12$ in \cite{CG1}), and the
Yang-Mills functional on $\mathcal{E}_2=p_2\mathcal{A}_\Phi^{q_m}$ is given by
\begin{eqnarray}\label{eqn 2}
\mathcal{YM}(\nabla_2) & = & \sum_{1\leq i<j\leq m}\tau_{q_m}([\nabla_{2i},\nabla_{2j}]^*[\nabla_{2i},\nabla_{2j}])
\end{eqnarray}
where, $\tau_{q_m}$ denotes the trace $\tau_\Phi\otimes Trace$ on $M_{q_m}(\mathcal{A}_\Phi)$. Now, $\mathcal{E}=\mathcal{E}_1\otimes\mathcal{E}_2=
(p_1\otimes p_2)\mathcal{A}_\Psi^{q_nq_m}$ and the Yang-Mills functional on $\mathcal{E}$ is given by
\begin{eqnarray}\label{eqn 3}
\mathcal{YM}(\nabla) & = & \sum_{1\leq i<j\leq m+n}\tau([\nabla_i,\nabla_j]^*[\nabla_i,\nabla_j])
\end{eqnarray}
where, $\nabla_k:\mathcal{E}\longrightarrow\mathcal{E}$ are $\mathbb{C}$-linear maps and $\tau$ denotes the trace $\tau_\Theta\otimes\tau_\Phi
\otimes Trace\,$ on $M_{q_nq_m}(\mathcal{A}_\Psi)$. Since, $\nabla=\nabla_1\otimes 1+1\otimes\nabla_2$ we have
\begin{align*}
\nabla_k(e_1\otimes e_2) &= \nabla_{1k}(e_1)\otimes e_2\,,\,\,1\leq k\leq n\\
\nabla_k(e_1\otimes e_2) &= e_1\otimes\nabla_{2,k-n}(e_2)\,,\,\,n+1\leq k\leq n+m
\end{align*}
Then,
\begin{eqnarray*}
&  & \sum_{1\leq i<j\leq m+n}[\nabla_i,\nabla_j]^*[\nabla_i,\nabla_j]\\
& = &\sum_{1\leq i<j\leq n}[\nabla_{1i},\nabla_{1j}]^*[\nabla_{1i},\nabla_{1j}]\otimes 1+\sum_{1\leq i<j\leq m}1\otimes[\nabla_{2i},\nabla_{2j}]^*
[\nabla_{2i},\nabla_{2j}]\,.
\end{eqnarray*}
In view of equations (\ref{eqn 1}),(\ref{eqn 2}),(\ref{eqn 3}), and because $\tau=\tau_\varTheta\otimes\tau_\Phi\otimes Trace$, we can now conclude
that the Yang-Mills functional is additive in this case, i,e. the condition described in Propn. (\ref{condition for additivity}) is satisfied.
\end{proof}
\bigskip


\section{The case of spin manifolds and Matrix algebras}

Let $\mathcal{M}$ be an even dimensional closed Riemannian spin manifold and $\mathcal{A}_1=C^\infty(\mathcal{M})$ be the algebra of smooth functions.
It is known that $C^\infty(\mathcal{M})$ is spectrally invariant in the unital $C^*$-algebra $C(\mathcal{M})$ (\cite{GVF}). Let $\mathcal{A}_2$
be a matrix algebra. Consider $\mathcal{A}=\mathcal{A}_1\otimes\mathcal{A}_2$. This algebra is a generalization of the product system ``four dimensional
manifold $\times$ 2-point space" considered in (\cite{Con2}). This is the algebra appearing in many examples in Physics. Let $\pi_1:C^\infty(\mathcal{M})
\longrightarrow\mathcal{B}(\mathcal{H}_1=L^2(S))$ be the representation of smooth functions on the square-integrable spinors, and $\pi_2:\mathcal{A}_2
\longrightarrow\mathcal{B}(\mathbb{C}^n)$ be a faithful representation of the matrix algebra $\mathcal{A}_2$ on $\mathbb{C}^n$ for some suitable $n$.
Let $D_1=i\slashed{\partial}_\mu\gamma^\mu$ be the Dirac operator associated to the spin manifold $\mathcal{M}$, and $D_2$ be a $n\times n$ self-adjoint
matrix. Let $\gamma$ denotes the grading automorphism of the Clifford algebra associated to $\mathcal{M}$ ($\gamma:=i^{d/2}\gamma_1\ldots\gamma_d$).
We have two even spectral triples $(C^\infty(\mathcal{M}),\mathcal{H}_1,D_1,\gamma)$ and $(\mathcal{A}_2,\mathbb{C}^n,D_2,\gamma_2)$. Consider
$D=D_1\otimes 1+\gamma\otimes D_2$. It is known that for all $k\geq 2$,
\begin{eqnarray}\label{an equality for function space}
\pi_1(\Omega^k(\mathcal{A}_1))\supseteq\pi_1(J^k(\mathcal{A}_1))=\pi_1(\Omega^{k-2}(\mathcal{A}_1))
\end{eqnarray}
and for all $k\geq 1,\,\Omega_{D_1}^{\,k}\cong\Gamma\left(\mathcal{M},\wedge^k\,T^*\mathcal{M}\right)$ (\cite{Con2},\cite{LAN}).

Now, for two even spectral triples $(\mathcal{A}_j,\mathcal{H}_j,D_j,\gamma_j),\,j=1,2$, there is an isomorphism of dgas between the Dirac dga
$\Omega_D^\bullet(\mathcal{A}_1\otimes \mathcal{A}_2)$ and the skew dga $\widetilde\Omega_D^\bullet(\mathcal{A}_1,\mathcal{A}_2)$ (see
\cite{KPPW} for detail). That is,
\begin{eqnarray}\label{isomorphism of tensored complex}
\Omega_D^n(\mathcal{A}_1\otimes \mathcal{A}_2) \cong \widetilde\Omega_D^n(\mathcal{A}_1,\mathcal{A}_2)\,\,\,\forall\,n\geq 0,
\end{eqnarray}
where the definition of $\widetilde\Omega_D^\bullet(\mathcal{A}_1,\mathcal{A}_2)$ is given below.

\begin{definition}\label{description of skew complex}
Consider the reduced universal dgas $\left(\Omega^\bullet(\mathcal{A}_1),d_1\right)$ and $\left(\Omega^\bullet(\mathcal{A}_2),d_2\right)$ associated
with the spectral triples $(\mathcal{A}_1,\mathcal{H}_1,D_1,\gamma_1)$ and $(\mathcal{A}_2,\mathcal{H}_2,D_2,\gamma_2)$ respectively. Now, consider
the product dga $\left(\Omega^\bullet(\mathcal{A}_1)\otimes\Omega^\bullet(\mathcal{A}_2)\,,\,\widetilde d\,\right)$ where,
\begin{align*}
(\omega_i\otimes u_j).(\omega_p\otimes u_q) &:= (-1)^{jp}\omega_i\omega_p\otimes u_ju_q\\
\widetilde{d}(\omega_i\otimes u_j) &:= d_1(\omega_i)\otimes u_j+(-1)^i\omega_i\otimes d_2(u_j)
\end{align*}
for $\omega_\bullet\in\Omega^\bullet(\mathcal{A}_1)$ and $u_\bullet\in\Omega^\bullet(\mathcal{A}_2)$. One can represent $\Omega^\bullet
(\mathcal{A}_1)\otimes\Omega^\bullet(\mathcal{A}_2)$ on $\mathcal{B}(\mathcal{H}_1\otimes\mathcal{H}_2)$ by the following map
\begin{center}
$\widetilde{\pi}(\omega_i\otimes u_j):=\pi_1(\omega_i)\gamma_1^j\otimes \pi_2(u_j)\,.$
\end{center}
Let 
\begin{center}
$\widetilde{J_0}^{k}:=Ker\left\{\widetilde\pi:\bigoplus_{i+j=k}\Omega^i(\mathcal{A}_1)\otimes\Omega^j(\mathcal{A}_2)\longrightarrow\mathcal{B}
(\mathcal{H}_1\otimes\mathcal{H}_2)\right\},$
\end{center}
and $\widetilde{J}^{\,n}=\widetilde{J_0}^n+\widetilde{d}\widetilde{J_0}^{n-1}$. Define $\widetilde{\Omega}_D^n(\mathcal{A}_1,\mathcal{A}_2):=
\frac{\bigoplus_{i+j=n} \Omega^i(\mathcal{A}_1)\otimes\Omega^j(\mathcal{A}_2)}{\widetilde J^n (\mathcal{A}_1,\mathcal{A}_2)}\,,\,\forall\,n\geq 0$.
We call it the skew dga.
\end{definition}

One has to compute $\Omega_{D_2}^\bullet(\mathcal{A}_2)$ first. Recall from ($\S\,(4)$ in \cite{KPPW}) that there are three cases. Let
$\mathcal{A}_2$ be given as the direct sum of the algebras $\mathcal{A}_{2,1}=M_p(\mathbb{C})$ and $\mathcal{A}_{2,2}=M_q(\mathbb{C})$.
The representation and the Dirac operator takes the form
\begin{center}
$\pi_2(\mathcal{A}_2)=\begin{pmatrix}
\mathcal{A}_{2,1} & 0\\
0 & \mathcal{A}_{2,2}
\end{pmatrix}\quad,\quad D_2=\begin{pmatrix}
0 & \mu^*\\
\mu & 0
\end{pmatrix}$
\end{center}
where $\mu$ denotes an arbitrary (non-zero) complex $p\times q$ matrix. Then, one has the following three cases.
\begin{enumerate}
\item[\underline{Case 1}~:] $\mu^*\mu\sim 1_{q\times q}$ and $\mu\mu^*\sim 1_{p\times p}$, which is possible only for $p=q$. In this case
$\mathcal{A}_{2,1}=\mathcal{A}_{2,2}$ and
\begin{center}
$\Omega_{D_2}^{2k}(\mathcal{A}_2)=\begin{pmatrix}
\mathcal{A}_{2,1} & 0\\
0 & \mathcal{A}_{2,1}
\end{pmatrix}\quad,\quad\Omega_{D_2}^{2k+1}(\mathcal{A}_2)=\begin{pmatrix}
0 & \mathcal{A}_{2,1}\\
\mathcal{A}_{2,1} & 0
\end{pmatrix}$
\end{center}
The multiplication rule is just the ordinary matrix multiplication of $2p\times 2p$ matrices. 
\item[\underline{Case 2}~:] $\mu^*\mu\nsim 1_{q\times q}$ and $\mu\mu^*\nsim 1_{p\times p}$. In this case
\begin{center}
$\Omega_{D_2}^1(\mathcal{A}_2)=\{X\in M_{q\times p}(\mathbb{C})\}\bigoplus\{Y\in M_{p\times q}(\mathbb{C})\}$
\end{center}
and there is no non-trivial multiplication of elements in $\Omega_{D_2}^1$.
\item[\underline{Case 3}~:] $q\leq p,\,\mu^*\mu\sim 1_{q\times q}$ and $\mu\mu^*\nsim 1_{p\times p}$. In this case
\begin{center}
$\Omega_{D_2}^1(\mathcal{A}_2)=\{X\in M_{q\times p}(\mathbb{C})\}\bigoplus\{Y\in M_{p\times q}(\mathbb{C})\}\quad,\quad\Omega_{D_2}^2
(\mathcal{A}_2)=\begin{pmatrix}
\mathcal{A}_{2,1} & 0\\
0 & 0
\end{pmatrix}$
\end{center} and all higher degrees of $\Omega_{D_2}^\bullet(\mathcal{A}_2)$ are trivial. The multiplication rule is given by
\begin{align*}
\Pi:\Omega_{D_2}^1(\mathcal{A}_2)\times\Omega_{D_2}^1(\mathcal{A}_2) &\longrightarrow \Omega_{D_2}^2(\mathcal{A}_2)\\
\begin{pmatrix}
X & 0\\
0 & Y
\end{pmatrix}\star\begin{pmatrix}
X^\prime & 0\\
0 & Y^\prime
\end{pmatrix} &= \begin{pmatrix}
X.Y^\prime & 0\\
0 & 0
\end{pmatrix}
\end{align*}
where, $X.Y^\prime$ denotes the usual matrix multiplication.
\end{enumerate}

Using these three cases and equation (\ref{an equality for function space}), it is shown in ($\S\,(7)$ in \cite{KPPW}) that the dga
$\widetilde{\Omega}_D^\bullet(\mathcal{A}_1,\mathcal{A}_2)$ is the tensor product of the Dirac dga of $\mathcal{A}_1$ and $\mathcal{A}_2$. From
the isomorphism in (\ref{isomorphism of tensored complex}) it now follows that Lemma (\ref{2 form for product system}) holds in this case, i,e.
our hypothesis in Section $(3)$ is justified.
\bigskip


\section{The case of quantum Heisenberg manifolds}

Recall the definition of quantum Heisenberg manifolds from (\cite{Rfl1}). For $x\in\mathbb{R},\,e(x)$ stands for $e^{2\pi ix}$, where $i=\sqrt{-1}$.
\begin{definition}
For any positive integer $c$, let $S^c$ denote the space of  smooth  functions $\Phi:\mathbb{R}\times\mathbb{T}\times\mathbb{Z}\rightarrow
\mathbb{C}$ such that 
\begin{itemize}
\item[(1)] $\Phi(x+k,y,p)=e(ckpy)\Phi(x,y,p)$ for all $k\in\mathbb{Z}$,
\item[(2)] For every polynomial $P$ on $\mathbb{Z}$ and every partial differential operator $\widetilde{X}=\frac{\partial^{m+n}}{\partial x^m\partial
y^n}$ on $\mathbb{R}\times\mathbb{T}$ the function $P(p)(\widetilde{X}\Phi)(x,y,p)$ is bounded on $K\times\mathbb{Z}$ for any compact subset $K$
of $\mathbb{R}\times\mathbb{T}$.
\end{itemize}
For each  $\hbar,\mu,\nu\in\mathbb{R},\mu^2+\nu^2\ne 0$, let ${\mathcal{A}}^{\infty}_{\hbar}$ denote $S^c$ with product and involution defined by
\begin{align*}
(\Phi\star\Psi)(x,y,p) &:= \sum_q\Phi(x-\hbar(q-p)\mu,y-\hbar(q-p)\nu,q)\Psi(x-\hbar q\mu,y-\hbar q\nu,p-q)\,,\\
\Phi^*(x,y,p) &:= \,\overline{\Phi}(x,y,-p)\,.
\end{align*}
Then, $\pi:{\mathcal{A}}^{\infty}_{\hbar}\rightarrow\mathcal{B}(L^2(\mathbb{R}\times\mathbb{T}\times\mathbb{Z}))$ given by
\begin{center}
$(\pi(\Phi)\xi)(x,y,p)=\sum_q\Phi(x-\hbar(q-2p)\mu,y-\hbar(q-2p)\nu,q)\xi(x,y,p-q)$
\end{center}
gives a faithful representation of the involutive  algebra $\AAZ$. Now, ${\mathcal{A}}^{c,\hbar}_{\mu,\nu}=$ norm closure of $\pi(\AAZ)$ is called
the quantum Heisenberg manifold.
\end{definition}

We will identify $\AAZ$ with $\pi(\AAZ)$ without any mention. Since, we are  going to work with fixed parameters $c,\mu,\nu,\hbar$ we will drop
them altogether and denote ${\mathcal{A}}^{c,\hbar}_{\mu,\nu}$ simply by $\mathcal{A}_\hbar$. Here the subscript remains merely as a reminiscent
of  Heisenberg only to distinguish it from a general algebra. Moreover, $\AAZ$ is spectrally invariant subalgebra of $\mathcal{A}_\hbar$.

{\bf Action of the Heisenberg group:} Let $c$ be a positive integer. Let us consider the  group structure on $G=\mathbb{R}^3=\{(r,s,t): r,s,t\in
\mathbb{R}\}$ given by the multiplication
\begin{center}
$(r,s,t)(r',s',t')=(r+r',s+s',t+t'+csr').$
\end{center}
There is an explicit isomorphism (\cite{CG2}) between $G$ and $H_3$, the Heisenberg group of $3\times 3$ upper triangular matrices with real entries
and $1$'s on the diagonal. For $\Phi\in S^c,(r,s,t)\in\mathbb{R}^3\equiv G$,  
\begin{center}
$(L_{(r,s,t)} \phi )(x,y,p)=e(p(t+cs(x-r)))\phi(x-r,y-s,p)$
\end{center}
extends to an ergodic action of the Heisenberg group on ${\mathcal{A}}_\hbar$.

{\bf The Trace:} The linear functional $\tau:\AAZ\rightarrow\mathbb{C}$, given by $\tau(\phi)=\int^1_0\int_{\mathbb{T}}\phi(x,y,0)dxdy$
is invariant under the Heisenberg group action. So, the group action can be lifted to $L^2(\AAZ)$. The action at the Hilbert space level is denoted
by the same symbol.

{\bf The G.N.S. Hilbert space:} The Hilbert space $L^2(\AAZ,\tau)$ obtained by applying the G.N.S. construction to $\tau$ is isomorphic to
$L^2(\mathbb{T}\times\mathbb{T}\times\mathbb{Z})\cong L^2([0,1]\times[0,1]\times\mathbb{Z})$ (\cite{CS}).

{\bf The spectral triple:} One fixes an inner product on the Lie algebra of the Heisenberg Lie group by declaring the following basis,
\begin{eqnarray*}
X_1=\left(\begin{matrix} 0&0&0 \cr 0&0&1 \cr 0& 0&0\cr \end{matrix}\right),
X_2=\left(\begin{matrix} 0&c&0 \cr 0&0&0\cr 0& 0&0\cr\end{matrix}\right),
X_3=\left(\begin{matrix} 0&0&c\alpha  \cr 0&0&0 \cr 0& 0&0\cr\end{matrix}\right)
\end{eqnarray*}
as orthonormal. Here, $\alpha>1$ is a real number. Then, $D_\hslash=\sum_{j=1}^3i\delta_j\otimes\sigma_j$ is an unbounded self-adjoint operator
with compact resolvent acting on $\mathcal{H}_\hslash:=L^2(\AAZ,\tau)\otimes\mathbb{C}^2$, where
\begin{align*}
\delta_1(f) &= - \frac{\partial f}{\partial x}\,,\\
\delta_2(f) &= 2\pi icpxf(x,y,p)-\frac{\partial f}{\partial y}\,,\\
\delta_3(f) &= 2\pi ipc\alpha f(x,y,p)\,.
\end{align*}
and
\begin{eqnarray*}
\sigma_1=\left(\begin{matrix} 1& 0 \cr 0&-1 \cr \end{matrix}\right)\,,
\,\,\sigma_2=\left(\begin{matrix} 0& -1 \cr -1& 0 \cr \end{matrix}\right)\,,
\,\,\sigma_3=\left(\begin{matrix} 0& i \cr -i& 0 \cr \end{matrix}\right)
\end{eqnarray*}
are the $2\times 2$ Pauli spin matrices. The tuple $(\AAZ,\mathcal{H}_\hslash,D_\hslash)$ is a $3$-summable spectral triple on $\AAZ$
(\cite{CS},\cite{CG2}).
\medskip

We will consider $\{1,\hbar\mu,\hbar\nu\}$ to be linearly independent over $\mathbb{Q}$. In that case $\mathcal{A}_\hbar$, and hence
$\AAZ$, becomes a simple algebra (\cite{Rfl1},\cite{CS}). We need this simpleness otherwise computation of the Dirac dga $\Omega_D^\bullet$
done in (\cite{CS}) fails. Let $\phi_{mn}\in S^c$ be the function $\phi_{m,n}(x,y,p)=e(mx+ny)\delta_{p0}$. These functions are eigenfunctions
for $\,\delta_j$'s and they satisfy
\begin{align*}
\delta_1(\phi_{10}) &= 2\pi\phi_{10} & \delta_2(\phi_{10}) &=0 & \delta_3(\phi_{10}) &=0\\
\delta_1(\phi_{01}) &=0 & \delta_2(\phi_{01}) &= 2\pi\phi_{01} & \delta_3(\phi_{01}) &=0
\end{align*}
Using these functions $\,\phi_{mn}$, and simpleness of $\AAZ$, it is shown in (\cite{CS}) that
\begin{align*}
\Omega_{D_\hslash}^1(\AAZ)=(\AAZ)^3\quad &, \quad\pi(\Omega^2(\AAZ))\cong(\AAZ)^4\,,\\
\pi(dJ_0^1(\AAZ))\cong\AAZ\quad &, \quad\Omega_{D_\hslash}^2(\AAZ)=(\AAZ)^3\,.
\end{align*}
In this section we consider $\mathcal{A}=\AAZ\otimes\AAZ$.
Since, $(\AAZ,\mathcal{H}_\hslash,D_\hslash)$ is an odd spectral triple we apply the process described in point (\ref{from odd to even}) in Section
$(2)$ to make it even with grading operator $\gamma$. Let $D=D_\hslash\otimes 1+\gamma\otimes D_\hslash$. Unlike the last section, here we go in
a straightforward way to verify that our hypothesis in Section $(3)$, or equivalently Lemma (\ref{2 form for product system}), is justified in this
case.

\begin{proposition}\label{to be used in next section}
$\Omega_D^2(\mathcal{A})\cong\,\Omega_{D_\hslash}^2(\AAZ)\otimes\AAZ\bigoplus\AAZ\otimes\Omega_{D_\hslash}^2(\AAZ)\bigoplus\Omega_{D_\hslash}^1(\AAZ)
\otimes\Omega_{D_\hslash}^1(\AAZ)$ as $\mathcal{A}=\AAZ\otimes\AAZ$-bimodules.
\end{proposition}
\begin{proof}
We first claim that $\pi_1(\Omega^2(\AAZ))\otimes\AAZ+\AAZ\otimes\pi_2(\Omega^2(\AAZ))$ is a free bimodule of rank 7 over $\AAZ\otimes\AAZ$. Note
that $\Omega_{D_\hslash}^1(\AAZ)\otimes\Omega_{D_\hslash}^1(\AAZ)$ is free of rank 9 and hence, in view of Lemma (\ref{numerator of 2nd form}), it
is enough to show that $\pi(\Omega^2(\mathcal{A}))$ is free with rank 16. Arbitrary element of $\pi(\Omega^2(\mathcal{A}))$ looks like
\begin{eqnarray*}
&  & \sum_{i,j,k}(a_{0i}\otimes a_{1i})[D,b_{0j}\otimes b_{1j}][D,c_{0k}\otimes c_{1k}]\\
& = & \sum_{i,j,k}a_{0i}\left(\sum_{n=1}^3\delta_n(b_{0j})\otimes\sigma_n\right)\left(\sum_{n=1}^3\delta_n(c_{0k})\otimes\sigma_n\right)\otimes
a_{1i}b_{1j}c_{1k}\\
&   & \quad+a_{0i}b_{0j}c_{0k}\otimes a_{1i}\left(\sum_{n=1}^3\delta_n(b_{1j})\otimes\sigma_n\right)\left(\sum_{n=1}^3\delta_n(c_{1k})\otimes
\sigma_n\right)\\
&   & \quad+\gamma\left(a_{0i}b_{0j}\left(\sum_{n=1}^3\delta_n(c_{0k})\otimes\sigma_n\right)\otimes a_{1i}\left(\sum_{n=1}^3\delta_n(b_{1j})\otimes
\sigma_n\right)c_{1k}\right)\\
&   & \quad-\gamma\left(a_{0i}\left(\sum_{n=1}^3\delta_n(b_{0j})\otimes\sigma_n\right)c_{0k}\otimes a_{1i}b_{1j}\left(\sum_{n=1}^3\delta_n(c_{1k})
\otimes\sigma_n\right)\right)\\
& = & \sum_{i,j,k}\left(a_{0i}\left(\sum_{n=1}^3\delta_n(b_{0j})\delta_n(c_{0k})\right)\otimes a_{1i}b_{1j}c_{1k}+a_{0i}b_{0j}c_{0k}\otimes a_{1i}
\left(\sum_{n=1}^3\delta_n(b_{1j})\delta_n(c_{1k})\right)\right)\otimes I_4\\
&   & \quad+\sum_{1\leq m<n\leq 3}(a_{0i}(\delta_m(b_{0j})\delta_n(c_{0k})-\delta_n(b_{0j})\delta_m(c_{0k}))\otimes a_{1i}b_{1j}c_{1k})
\otimes(\sigma_m\sigma_n\otimes I_2)\\
&   & \quad+\sum_{1\leq m<n\leq 3}(a_{0i}b_{0j}c_{0k}\otimes a_{1i}(\delta_m(b_{1j})\delta_n(c_{1k})-\delta_n(b_{1j})\delta_m(c_{1k})))
\otimes(I_2\otimes\sigma_m\sigma_n)\\
&   & \quad+\,\gamma\sum_{1\leq m,n\leq 3}(a_{0i}b_{0j}\delta_m(c_{0k})\otimes a_{1i}\delta_n(b_{1j})c_{1k}-a_{0i}\delta_m(b_{0j})c_{0k}
\otimes a_{1i}b_{1j}\delta_n(c_{1k}))\otimes(\sigma_m\otimes\sigma_n)
\end{eqnarray*}
Here, we are using the canonical isomorphism
\begin{center}
$\left(L^2(\AAZ,\tau)\otimes\mathbb{C}^2\right)\otimes\left(L^2(\AAZ,\tau)\otimes\mathbb{C}^2\right)\longrightarrow\left(L^2(\AAZ,\tau)\otimes
L^2(\AAZ,\tau)\right)\otimes\left(\mathbb{C}^2\otimes\mathbb{C}^2\right)$
\end{center}
of Hilbert spaces to push all the matrices $\{I_4=I_2\otimes I_2\,,\sigma_m\sigma_n\otimes I_2\,,I_2\otimes\sigma_m\sigma_n\,,\,\sigma_m\otimes
\sigma_n\}$ to the extreme right. Observe that $\{I_4\,,\sigma_m\sigma_n\otimes I_2\,,I_2\otimes\sigma_m\sigma_n\,,\,\sigma_j\otimes\sigma_\ell\,:\,1
\leq m<n\leq 3\,,\,1\leq j,\ell\leq 3\}$ is a linear basis of $M_4(\mathbb{C})=M_2(\mathbb{C})\otimes M_2(\mathbb{C})$. Thus, we get an obvious
injective bimodule map $\pi(\Omega^2(\mathcal{A}))\longrightarrow(\AAZ\otimes\AAZ)^{16}$. We claim that this map is onto. For that first consider
the following three elements of $\pi(\Omega^2(\mathcal{A}))$ given respectively by setting
\begin{center}
$a_{1i}=b_{1j}=c_{1k}=1\quad;\quad a_{0i}=b_{0j}=c_{0k}=1\quad;\quad b_{0j}=c_{1k}=1$
\end{center}
for all $j,k$. Now, use the simpleness of $\AAZ$ and follow the proof of Proposition $[21]$ in (\cite{CS}) (the proof of the fact that $\pi
(\Omega^1(\AAZ))=(\AAZ)^3$ and $\pi(\Omega^2(\AAZ))=(\AAZ)^4$). So, we conclude that $\pi_1(\Omega^2(\AAZ))\otimes\AAZ+\AAZ\otimes\pi_2
(\Omega^2(\AAZ))$ is a free bimodule of rank 7 over $\AAZ\otimes\AAZ$.

Now, we show that $\pi_1(d_1J_0^1(\AAZ))\otimes\AAZ\bigcap\AAZ\otimes\pi_2(d_2J_0^1(\AAZ))$ is a free $\AAZ\otimes\AAZ$-module of rank one.
Since, $\pi_j(d_jJ_0^1(\AAZ))\cong\AAZ$ for $j=1,2$, we only need to prove that $\AAZ\otimes\AAZ\subseteq\pi_1(d_1J_0^1(\AAZ))\otimes\AAZ
\bigcap\AAZ\otimes\pi_2(d_2J_0^1(\AAZ))$, the other inclusion being obvious. For any $\xi=\sum_ja_j\otimes b_j\in\AAZ\otimes\AAZ$, we can write
$\xi=\sum_j\omega_j\otimes b_j=\sum_ja_j\otimes v_j$ with each $\omega_j\in\pi_1(d_1J_0^1(\AAZ))$ and $v_j\in\pi_2(d_2J_0^1(\AAZ))$. This is
because $\pi_j(d_jJ_0^1(\AAZ))\cong\AAZ$ for $j=1,2$. Hence, $\xi\in\pi_1(d_1J_0^1(\AAZ))\otimes\AAZ\bigcap\AAZ\otimes\pi_2(d_2J_0^1(\AAZ))$
and this concludes the claim. Hence, we have the following canonical isomorphism
\begin{eqnarray*}
&  & \pi_1(d_1J_0^1(\AAZ))\otimes\AAZ+\AAZ\otimes\pi_2(d_2J_0^1(\AAZ))\\
& \cong & \frac{\pi_1(d_1J_0^1(\AAZ))\otimes\AAZ\bigoplus\AAZ\otimes\pi_2(d_2J_0^1(\AAZ))}{\pi_1(d_1J_0^1(\AAZ))\otimes\AAZ\bigcap
\AAZ\otimes\pi_2(d_2J_0^1(\AAZ))}\\
& \cong & \frac{\AAZ\otimes\AAZ\bigoplus\AAZ\otimes\AAZ}{\AAZ\otimes\AAZ}\\
& \cong & \AAZ\otimes\AAZ
\end{eqnarray*}
of $\AAZ\otimes\AAZ$-bimodules. By Lemma (\ref{denominator of 2nd form}), $\pi(dJ_0^1(\mathcal{A}))$ now becomes a free bimodule of rank 1 over
$\AAZ\otimes\AAZ$. Since, $\Omega_{D_\hslash}^2(\AAZ)\otimes\AAZ\bigoplus\AAZ\otimes\Omega_{D_\hslash}^2(\AAZ)$ is also a free bimodule
of rank 6 over $\AAZ\otimes\AAZ$, we have a canonical isomorphism
\begin{center}
$\frac{\pi_1(\Omega^2(\AAZ))\otimes\AAZ+\AAZ\otimes\pi_2(\Omega^2(\AAZ))}{\pi(dJ_0^1(\mathcal{A}))}\cong\Omega_{D_\hslash}^2(\AAZ)
\otimes\AAZ\bigoplus\AAZ\otimes\Omega_{D_\hslash}^2(\AAZ)$
\end{center}
of $\mathcal{A}$-bimodules. Since, $\Omega_D^2(\mathcal{A})=\pi(\Omega^2(\mathcal{A}))/\pi(dJ_0^1(\mathcal{A}))$, final conclusion follows from
Lemma (\ref{numerator of 2nd form}). 
\end{proof}
\bigskip


\section{The case of noncommutative tori and quantum Heisenberg manifolds}

In this section we consider $\mathcal{A}_\Theta\otimes\AAZ$, where $\mathcal{A}_\Theta$ is a noncommutative $n$-torus and $\AAZ$ is a quantum
Heisenberg manifold. Recall from (\cite{CG1},\cite{CG2},\cite{CS}),
\begin{align*}
\Omega_{D_\Theta}^1(\mathcal{A}_\Theta)=\mathcal{A}_\Theta^n\quad &, \quad\Omega_{D_\Theta}^2(\mathcal{A}_\Theta)=\mathcal{A}_\Theta^{n(n-1)/2}\,,\\
\Omega_{D_\hslash}^1(\AAZ)=(\AAZ)^3\quad &, \quad\Omega_{D_\hslash}^2(\AAZ)=(\AAZ)^3\,.
\end{align*}
If $(\mathcal{A}_\Theta,\mathcal{H}_\Theta,D_\Theta)$ is not an even spectral triple (unless $n$ is even) then we apply the process described in
point (\ref{from odd to even}) in Section $(2)$ to make it even with grading operator $\gamma$. Let $D=D_\Theta\otimes 1+\gamma\otimes D_\hslash$
and $\mathcal{A}=\mathcal{A}_\Theta\otimes\AAZ$. In this section also we assume that $\{1,\hbar\mu,\hbar\nu\}$ is linearly independent over
$\mathbb{Q}$ so that $\AAZ$ is a simple algebra. Next Proposition shows that our hypothesis in Section $(3)$, or equivalently Lemma
(\ref{2 form for product system}), holds in this case also.

\begin{proposition}
$\Omega_D^2(\mathcal{A})\cong\,\Omega_{D_\Theta}^2(\mathcal{A}_\Theta)\otimes\AAZ\bigoplus\mathcal{A}_\Theta\otimes\Omega_{D_\hslash}^2(\AAZ)
\bigoplus\Omega_{D_\Theta}^1(\mathcal{A}_\Theta)\otimes\Omega_{D_\hslash}^1(\AAZ)$ as $\mathcal{A}=\mathcal{A}_\Theta\otimes\AAZ$-bimodules.
\end{proposition}
\begin{proof}
We only sketch the proof as computations are similar to the previous section. First note that $\Omega_{D_\Theta}^2(\mathcal{A}_\Theta)\otimes
\AAZ\bigoplus\mathcal{A}_\Theta\otimes\Omega_{D_\hslash}^2(\AAZ)$ is a free $\mathcal{A}$-bimodule of rank $\,n(n-1)/2+3$, and
$\Omega_{D_\Theta}^1(\mathcal{A}_\Theta)\otimes\Omega_{D_\hslash}^1(\AAZ)$ is free of rank $3n$. It can be shown that
\begin{center}
$\pi_1(d_1J_0^1(\mathcal{A}_\Theta))\otimes\AAZ+\mathcal{A}_\Theta\otimes\pi_2(d_2J_0^1(\AAZ))\cong\mathcal{A}_\Theta\otimes\AAZ$
\end{center}
as $\mathcal{A}$-bimodule, and hence by Lemma (\ref{denominator of 2nd form}), $\pi(dJ_0^1(\mathcal{A}))$ is a free bimodule of rank 1 over
$\mathcal{A}$. Hence, we need to show that $\pi(\Omega^2(\mathcal{A}))$ is a free $\mathcal{A}$-bimodule of rank $3n+4+n(n-1)/2$. As done in
the proof of Proposition (\ref{to be used in next section}), by writing down any arbitrary element of $\pi(\Omega^2(\mathcal{A}))$ explicitly,
one can observe that the claim follows similarly by using the simpleness of $\AAZ$ and the proof of Proposition $[5.3]$ in (\cite{CG1}) (the
proof of the fact that $\Omega_{D_\Theta}^2(\mathcal{A}_\Theta)=\mathcal{A}_\Theta^{n(n-1)/2}$ for any $n$-torus $\mathcal{A}_\Theta$).
\end{proof}
\bigskip

\section*{Acknowledgement}
Author gratefully acknowledges financial support of DST, India through INSPIRE Faculty award (Award No. DST/INSPIRE/04/2015/000901).

\end{document}